\documentclass[10pt]{article}
\usepackage[english]{babel}
\usepackage{amsmath,amsthm}
\usepackage{amsfonts}
\usepackage{enumerate}
\usepackage{url}
\usepackage{graphicx}
\usepackage{color}

\usepackage[top=1.43in, bottom=1.42in, left=1.75in, right=1.75in]{geometry}

\newtheorem{thm}{Theorem}[section]

\newtheorem{lem}[thm]{Lemma}
\newtheorem{prop}[thm]{Proposition}
\newtheorem{bigthm}{Theorem}

\theoremstyle{definition}

\newtheorem*{kohnconj}{Kohn's Conjecture}
\newtheorem{ques}{Question}
\newtheorem{conj}[ques]{Conjecture}

\theoremstyle{remark}
\newtheorem{rem}[thm]{Remark}
\newtheorem{example}[thm]{Example}
\numberwithin{equation}{section}
\newcommand{\norm}[1]{\lVert #1 \rVert^2}

\DeclareMathOperator*{\supp}{supp}
\DeclareMathOperator*{\suppp}{supp^+}
\DeclareMathOperator*{\suppm}{supp^{-}}

\newcommand{\spin}{\ifmmode{\rm Spin}\else{${\rm spin}$\ }\fi}
\newcommand{\spinc}{\ifmmode{{\rm Spin}^c}\else{${\rm spin}^c$\ }\fi}

\newcommand{\spincs}{\mathfrak s}

\begin{document}

\title{Alternating knots with unknotting number one}%
\author{Duncan McCoy\\
{\small University of Glasgow}\\
{\small\texttt{d.mccoy.1@research.gla.ac.uk}}}%
\date{April 2014}
% ----------------------------------------------------------------
\maketitle
\begin{abstract}
We prove that if an alternating knot has unknotting number one, then there exists an unknotting crossing in any alternating diagram. This is done by showing that the obstruction to unknotting number one developed by Greene in his work on alternating 3-braid knots is sufficient to identify all unknotting number one alternating knots. As a consequence, we also get a converse to the Montesinos trick: an alternating knot has unknotting number one if its branched double cover arises as half-integer surgery on a knot in $S^3$. We also reprove a characterisation of almost-alternating diagrams of the unknot originally due to Tsukamoto.
\end{abstract}
% ----------------------------------------------------------------
\section{Introduction}
Given a knot $K\subset S^3$, its {\em unknotting number}, $u(K)$, is a classical knot invariant going back to the work of Tait in the 19th century \cite{tait1876knots}. It is defined to be the minimal number of crossing changes required in any diagram of $K$ to obtain the unknot. Upper bounds for the unknotting number are easy to obtain, since one can take some diagram and find a sequence of crossing changes giving the unknot. It is far harder to establish effective lower bounds for the unknotting number, as it is not generally known which diagrams will exhibit the actual unknotting number \cite{bleiler1984note, bernhard1994unknotting, jablan1998unknotting}. One classical lower bound is the signature of a knot as defined by Trotter \cite{trotter1962homology}, which satisfies $|\sigma(K)|\leq 2 u(K)$ \cite{murasugi1965certain}. This is a particularly useful bound, since it may be computed in a variety of ways \cite{trotter1962homology, gordon1978signature}. Other bounds and obstructions have been constructed through the use of various knot-theoretic and topological invariants, including, among others, the Alexander module \cite[Theorem 7.10]{lickorish1997introduction}, the Jones polynomial \cite{stoimenow2004polynomial}, and the intersection form of 4-manifolds \cite{cochran1986unknotting, owens2008unknotting}.

The case of unknotting number one has been particularly well-studied. Recall that a {\em minimal diagram} for a knot is one containing the minimal possible number of crossings. Kohn made the following conjecture regarding unknotting number one knots and their minimal diagrams.

\begin{kohnconj}[Conjecture 12, \cite{kohn1991two}]
If $K$ is a knot with $u(K)=1$, then it has an unknotting crossing in a minimal diagram.
\end{kohnconj}

This has been resolved in a number of cases. The two-bridge knots with unknotting number one were classified by Kanenobu and Murakami \cite{Unknottwobridge}, using the Cyclic Surgery Theorem \cite{cglscyclic}. For alternating large algebraic knots, the conjecture was settled by Gordon and Luecke \cite{gordon2006knots}. Most recently, the conjecture was proved for alternating 3-braid knots by Greene \cite{Greene3Braid}, using a refined version of obstructions first developed from Heegaard Floer homology by  Ozsv{\'a}th and Szab{\'o} \cite{ozsvath2005knots}.

Our main result is the following.

\begin{bigthm}\label{thm:main}
For an alternating knot, $K$, the following are equivalent:
\begin{enumerate}[(i)]
\item $u(K)=1$;
\item The branched double cover, $\Sigma(K)$, can be obtained by half-integer surgery on a knot in $S^3$;
\item $K$ has an unknotting crossing in any alternating diagram.
\end{enumerate}
\end{bigthm}
Since the minimal diagrams of alternating knots are alternating, it resolves Kohn's conjecture for alternating knots.

\paragraph{}In general, Kohn's Conjecture seems somewhat optimistic. For example, there are 14-crossing knots with unknotting number one with minimal diagrams not containing an unknotting crossing \cite{stoimenow2001examples}. However these examples are not sufficient to disprove the conjecture, since they do have some minimal diagram with an unknotting crossing.

\paragraph{}An {\em almost-alternating diagram} is one which is obtained by a single crossing change in an alternating diagram. Theorem~\ref{thm:main} can be interpreted as showing that understanding alternating knots with unknotting number one is equivalent to understanding almost-alternating diagrams of the unknot. The methods in this paper allow us to give a proof of a characterisation of almost alternating diagrams of the unknot originally due to Tsukamoto \cite{tsukamoto2009almost}. A {\em reduced} diagram is one not containing any {\em nugatory crossings} (See Figure~\ref{fig:nugatory}).
 \begin{figure}[h]
  \centering
  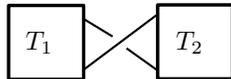
 \caption{A nugatory crossing}
 \label{fig:nugatory}
\end{figure}
Any reduced almost-alternating diagram of the unknot can be built up using only certain types of isotopies: {\em flypes} are illustrated in Figure~\ref{fig:flypedef}; and {\em tongue} and {\em twirl} moves, which are the inverses of the {\em untongue} and {\em untwirl} moves depicted in Figure~\ref{fig:unswirluntongue}. The basic almost-alternating diagrams from which all others are built are shown in Figure~\ref{fig:claspdiagram}, and denoted by $\mathcal{C}_m$.

\begin{bigthm}\cite[Corollary 1.1]{tsukamoto2009almost}\label{thm:tsukamoto}
Any reduced almost-alternating diagram of the unknot can be obtained from $\mathcal{C}_m$ for some non-zero integer $m$, by a sequence of flypes, tongue moves and twirl moves.
\end{bigthm}

Our proof is of very different flavour to the original which employed spanning surfaces and geometric arguments, rather than the machinery of Heegaard Floer homology and the topology of 4-manifolds. Together Theorem~\ref{thm:main} and Theorem~\ref{thm:tsukamoto} may be viewed as a complete description of alternating knots with unknotting number one.

\begin{figure}[h!]
  \centering
  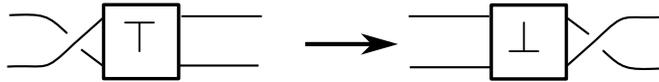
 \caption{A flype. It is known that any two reduced alternating diagrams of a knot can be related by a sequence of flypes \cite{Menasco93classification}.}
 \label{fig:flypedef}
\end{figure}

\begin{figure}[h!]
  \centering
  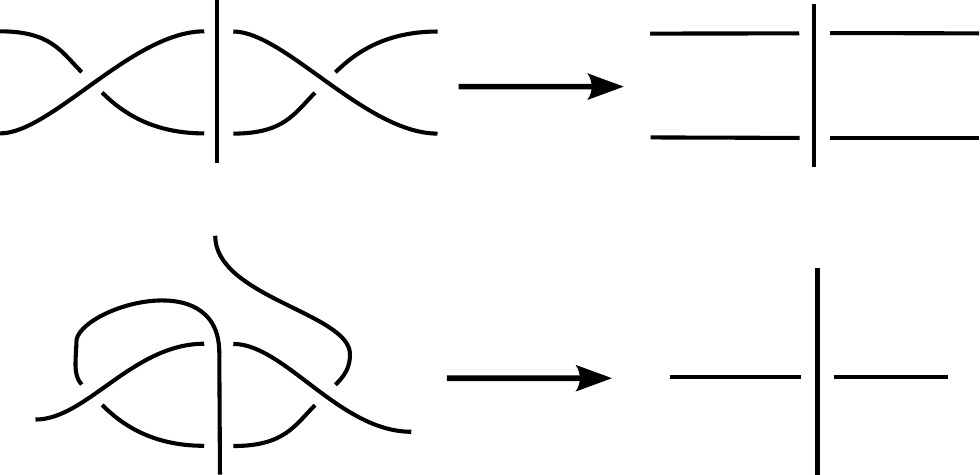
 \caption{An untongue move (top) and an untwirl (bottom). Their inverses are known as tongue and twirl moves respectively.}
 \label{fig:unswirluntongue}
\end{figure}

\begin{figure}[h!]
  \centering
  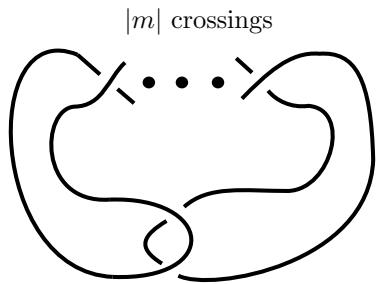
 \caption{The diagram $\mathcal{C}_m$, where the $|m|$ crossings are positive if $m>0$ and negative otherwise.}
 \label{fig:claspdiagram}
\end{figure}

\subsection{Outline of the proof}
The implication $(i)\Rightarrow (ii)$ is the Montesinos trick \cite{montesinos1973variedades} and $(iii)\Rightarrow (i)$ is trivial. We need to establish $(ii)\Rightarrow (iii)$. The minimal diagrams of an alternating knot are precisely the reduced alternating ones \cite{kauffman87state, murasugi86jones, thistlethwaite88alternating}. Thus it suffices to consider only reduced alternating diagrams, since a nugatory crossing cannot be an unknotting crossing. The starting point will be the obstruction to half-integer surgery of Greene, \cite[Theorem 4.5]{Greene3Braid}, which is a combination of Donaldson's Theorem A \cite{donaldson1983application} with restrictions arising from the $d$-invariants of Heegaard Floer theory. The version we will use is stated here as Theorem~\ref{thm:GreeneLspaceversion}. It shows that if the branched double cover of an alternating knot, $K$, arises as half-integer surgery on a knot in $S^3$, then, up to reflection of $K$, the lattice associated to the Goeritz matrix of an alternating diagram is a change-maker lattice. Change-maker lattices will be fully defined in Section~\ref{sec:CMlattices}. We will use the change-maker condition to establish the existence of an unknotting crossing.

\paragraph{}Taking further inspiration from Greene \cite{Greene3Braid}, the key definition we develop is that of a marked crossing, which we will describe briefly here. A change-maker lattice, $L=\langle \sigma, \rho \rangle^{\bot} \subseteq \mathbb{Z}^{r+2}$, is an orthogonal complement of two vectors, $\sigma$ and $\rho$, where $\sigma$ satisfies the change-maker condition and $\norm{\rho}=2$. If $\{e_{-1}, \dotsc, e_r\}$ is an orthonormal basis for $\mathbb{Z}^{r+2}$, then we can we can assume $\rho=e_0-e_{-1}$. Thus, any $v\in L$ satisfies $v\cdot e_0=v\cdot e_{-1}$.
Given an alternating diagram, $D$, with a chessboard colouring, let $\Lambda_D$ be the lattice defined by the associated Goeritz matrix. An isomorphism between $\Lambda_D$ and $L$ allows us to label each unshaded region of $D$ with a vector in $L$. We say that a crossing $c$ is {\em marked} if it occurs between regions labeled by vectors of the form $v_1+e_0+e_{-1}$ and $v_2-e_0-e_{-1}$, with $v_1\cdot e_0=v_2\cdot e_0=0$, as illustrated in Figure~\ref{fig:markedcrossing}.

\begin{figure}[h!]
  \centering
  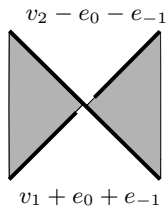
 \caption{A marked crossing.}
 \label{fig:markedcrossing}
\end{figure}
To complete the proof of the implication $(ii)\Rightarrow (iii)$, we show that marked crossings exist and that they are unknotting crossings.

\begin{bigthm}\label{thm:markedmeansunknotting}
Let $D$ be a reduced alternating diagram, and suppose that the lattice $\Lambda_D$ can be embedded into $\mathbb{Z}^{r+2}$ as a change-maker lattice. Then there is at least one marked crossing in $D$, and any marked crossing is an unknotting crossing.
\end{bigthm}

We establish Theorem~\ref{thm:markedmeansunknotting} by a series of careful arguments which exploit the fact that $\Lambda_D$ is both a graph lattice and a change-maker lattice. The necessary features of graph and change-maker lattices are explored in Sections~\ref{sec:graphlattices} and \ref{sec:CMlattices}, respectively. By combining properties of these two types of lattice we are able to place strong restrictions on the embedding of $\Lambda_D$ into $\mathbb{Z}^{r+2}$. Using Lemma~\ref{lem:flype}, we will use the properties of $\Lambda_D$ to deduce the existence of flypes in $D$ algebraically. In particular, we will find a sequence of flypes fixing the marked crossings which give a diagram, $\hat{D}$, with a nice embedding of $\Lambda_{\hat{D}}$ into $\mathbb{Z}^{r+2}$. Since flypes preserve unknotting crossings, we can use properties of $\hat{D}$ to prove Theorem~\ref{thm:markedmeansunknotting} for the original diagram $D$.

The existence of marked crossings is established in Section~\ref{sec:markedcrossings}, where we also show that there are only ever two regions between which marked crossings can be found. Let $D'$ be the almost-alternating diagram obtained by changing a marked crossing and let $K'$ be the corresponding knot. It can be shown that $K'$ has determinant 1. If there is more than one marked crossing, this will be sufficient to show that $K'$ is the unknot. In the case of a general alternating knot, we must work harder to show that $K'$ is the unknot. The motivation for this part of the proof comes from the results of Tsukamoto \cite{tsukamoto2009almost}.

In Section~\ref{sec:singlemarkedc}, we show there is a sequence of flypes to a diagram which admits an untongue or untwirl move after changing the marked crossing. Changing the marked crossing and performing this move gives a diagram $\widetilde{D}'$ for $K'$. We prove inductively that $K'$ is the unknot by showing that $\widetilde{D}'$ can be obtained by changing a marked crossing in a smaller alternating knot diagram $\widetilde{D}$. This is the work of Section~\ref{sec:smallerdiagram}.

The programme outlined above provides a proof of the following result from which Theorem~\ref{thm:main} easily follows.
\begin{bigthm}\label{thm:technical}
Let $K$ be an alternating knot with reduced alternating diagram $D$ and $\det K =d$. The following are equivalent:
\begin{enumerate}[(i)]
\item $K$ can be unknotted by changing a negative crossing if $\sigma(K)=0$ or positive crossing if $\sigma(K)=-2$;
\item there is a knot $\kappa \subset S^3$ such that the branched double cover $\Sigma(K)$ arises as the half-integer surgery $S^3_{-\frac{d}{2}}(\kappa)$;
\item the Goeritz lattice $\Lambda_D$ is isomorphic to a change-maker lattice;
\item there is an unknotting crossing in $D$ which is negative if $\sigma(K)=0$ and positive crossing if $\sigma(K)=-2$.
\end{enumerate}
\end{bigthm}

The proof of Theorem~\ref{thm:markedmeansunknotting} shows that the almost-alternating diagram of the unknot obtained by changing a marked crossing can be reduced by flypes, untongue and untwirl moves to one of the $\mathcal{C}_m$. Thus when we show in Section~\ref{sec:aadiagrams} that any unknotting crossing can be made into a marked crossing for some embedding of the Goeritz form as a change-maker lattice, we obtain a proof of Theorem~\ref{thm:tsukamoto}.

\paragraph{} In Section~\ref{sec:furtherqus}, we comment briefly on the implications of this work and
some of the further questions it raises.

\subsection{Acknowledgements}
The author would like to thank his supervisor, Brendan Owens, for his patient and careful reading of this paper, and for providing a great deal of useful advice. He is grateful to Liam Watson for helpful discussions and for suggesting Example~\ref{exam:montyconverse}. He wishes to thank Joshua Greene for helpful correspondence and acknowledge the influential role of his work, in particular \cite{Greene3Braid} and \cite{GreeneLRP}, which provided inspiration and motivation for many of the ideas in this work. He would also like to thank Tara Brendle and Ana Lecuona for helpful comments on this paper.

\section{Lattices}\label{sec:lattices}
In this section, we develop the lattice theoretic results on which the proofs of the main results rely. The two types of lattice we will require are graph lattices and change-maker lattices, which are discussed in Sections~\ref{sec:graphlattices} and \ref{sec:CMlattices} respectively.

\paragraph{}A {\em positive-definite integral lattice}, $L$, is a finitely-generated free abelian group with a positive-definite, symmetric form $L\times L \rightarrow \mathbb{Z}$, given by $(x,y)\mapsto x\cdot y$. All lattices occurring in this paper will be of this form. We say that $L$ is {\em indecomposable} if it cannot be written as an orthogonal direct sum $L= L_1 \oplus L_2$ with $L_1,L_2$ non-zero sublattices.

For $v \in L$, its {\em norm} is the quantity $\norm{v}=v\cdot v$. We say that $v$ is {\em irreducible} if we cannot find $x,y \in L \setminus \{0\}$ such that $v=x+y$ and $x\cdot y \geq 0$.

\subsection{Graph lattices}\label{sec:graphlattices}
In this section, we collect the necessary results on graph lattices. The material in this section is largely based on that of Greene \cite[Section 3.2]{GreeneLRP}. There has, however, been some reworking since we find it convenient to avoid using the concept of a root vertex.

Let $G=(V,E)$ be a finite, connected, undirected graph with no self-loops. For a pair of disjoint subsets $R,S \subset V$, let $E(R,S)$ be the set of edges between $R$ and $S$. Define $e(R,S)=|E(R,S)|$. We will use the notation $d(R)=e(R,V\setminus R)$.

\paragraph{} Let $\overline{\Lambda}(G)$ be the free abelian group generated by $v\in V$. Define a symmetric bilinear form on $\overline{\Lambda}(G)$ by
\[
v\cdot w =
  \begin{cases}
   d(v)            & \text{if } v=w \\
   -e(v,w)       & \text{if } v\ne w.
  \end{cases}
\]
In this section we will use the notation $[R]=\sum_{v\in R}v$, for $R\subseteq V$. The above definition gives
\[
v\cdot [R] =
  \begin{cases}
   -e(v,R)            & \text{if } v\notin R \\
   e(v,V\setminus R)       & \text{if } v\in R.
  \end{cases}
\]

From this it follows that $[V]\cdot x= 0$ for all $x \in \overline{\Lambda}(G)$. We define the {\em graph lattice} of $G$ to be
$$\Lambda(G):= \frac{\overline{\Lambda}(G)}{\mathbb{Z}[V]}.$$
The bilinear form on $\overline{\Lambda}(G)$ descends to $\Lambda(G)$. Since we have assumed that $G$ is connected, the pairing on $\Lambda(G)$ is positive-definite. This makes $\Lambda(G)$ into an integral lattice. Henceforth, we will abuse notation by using $v$ to denote both the vertex $v \in V$ and its image in $\Lambda(G)$.

\paragraph{}We compute the product of arbitrary $x,y \in \Lambda(G)$. Let $x=\sum_{v\in V} b_v v$ and $y=\sum_{v\in V} c_v v$ be elements of $\Lambda(G)$. For any $v\in V$,
$$v\cdot y=c_v d(v) - \sum_{u\ne v} e(v,u)c_u = \sum_{u \in V} (c_v-c_u)e(v,u).$$
Therefore,
$$x \cdot y = \sum_{v \in V}b_v \sum_{u\in V}(c_v-c_u)e(v,u)=\sum_{u,v \in V}b_v(c_v-c_u)e(v,u).$$
Since we also have
$$x \cdot y =\sum_{u,v \in V}c_u(b_u-b_v)e(v,u),$$
we can express the pairing $x\cdot y$ as
\begin{equation}\label{eq:prod1}
x \cdot y = \frac{1}{2}\sum_{u,v \in V}(c_v-c_u)(b_v-b_u)e(v,u).
\end{equation}
If $y=[R]-x$ for some $R\subseteq V$, then
\[
c_v =
  \begin{cases}
   1-b_v            & \text{if } v\in R \\
   -b_v       & \text{if } v\notin R.
  \end{cases}
\]
and \eqref{eq:prod1} shows

\begin{equation}\label{eq:usefulformula}
([R]-x)\cdot x =
\sum_{u\in R, v\in V\setminus R}b_{v,u} (1-b_{v,u})e(v,u)
-\frac{1}{2}\sum_{u,v\in R}b_{v,u}^2 e(v,u)
-\frac{1}{2}\sum_{u,v \in V\setminus R}b_{v,u}^2 e(v,u),
\end{equation}
where $b_{v,u}=b_v-b_u$. Examining each term in \eqref{eq:usefulformula} individually, we see that the right hand side is at most zero. This inequality will be used so often that we will record it as the following lemma.
\begin{lem}\label{lem:usefulbound}
Let $x=[R]$ be a sum of vertices, then for any $z\in \Lambda(G)$, we have
$$(x-z)\cdot z\leq 0.$$
\end{lem}
The irreducible vectors in $\Lambda(G)$ will be of particular interest.
\begin{lem}\label{lem:irreducible}
 If $x \in \Lambda(G)\setminus \{0\}$, then $x$ is irreducible if, and only if, $x=[R]$ for some $R\subseteq V$ such that $R$ and $V\setminus R$ induce connected subgraphs of $G$.
\end{lem}
\begin{proof}
Write $x=\sum_{v\in V} a_v v$. Since $[V]=0$, we can assume the $a_v$ are chosen so that $\min_{v\in V} a_v = 0$.

Suppose $x\ne 0$ is irreducible. Let $a=\max_{v \in V}a_v\geq 1$ and $R=\{v|a_v=a\}$.
By direct computation, we get
\begin{align*}
(x-[R])\cdot [R] &= \sum_{v\in R} \bigg((a-1)d(v) -\sum_{u\in V\setminus \{v\}}a_u e(u,v)\bigg)\\
&\geq (a-1)\sum_{v\in R} \bigg(d(v) -\sum_{u\in V\setminus \{v\}}e(u,v)\bigg)\\
&= 0.
\end{align*}
Combining this with the irreducibility of $x$, it follows that $x=[R]$. If the subgraph induced by $R$ is not connected there would be $S,T$ disjoint with $R=S\cup T$ and $e(S,T)=0$, but this would give $x=[S]+[T]$ and $[S]\cdot [T]=0$. This also shows the subgraph induced by $V\setminus R$ must be connected, since $x$ irreducible implies that $-x=[V]-x$ is irreducible.
\paragraph{} To show the converse, we apply \eqref{eq:usefulformula} in the case where $R$ and $V\setminus R$ induce connected subgraphs. Let $y= \sum_{v\in V} c_v v$ be such that $([R]-y)\cdot y\geq 0$. By Lemma~\ref{lem:usefulbound}, this means $([R]-y)\cdot y = 0$. In particular, using (\ref{eq:usefulformula}), we have
\begin{equation*}
([R]-y)\cdot y=\sum_{u\in R, v \in V\setminus R} c_{u,v}(1-c_{u,v})e(u,v)-\frac{1}{2}\sum_{u,v \in V\setminus R}c_{u,v}^2e(u,v)-\sum_{u,v \in R}c_{u,v}^2 e(u,v)=0,
\end{equation*}
where $c_{u,v}=c_u-c_v$. This means every summand in the above equation must be 0. Thus if $u$ and $v$ are both in $R$ or $V\setminus R$ and $e(u,v)>0$, then $c_u=c_v$. Thus since $R$ and $V\setminus R$ induce connected subgraphs, $c_v$ is constant on $R$ and $V\setminus R$. If $u\in R$ and $v\in V \setminus R$, then $c_u=c_v$ or $c_u=1+c_v$. Thus if $c=c_v$ for some $v \in V\setminus R$, then $y=c[V]=0$ or $y=c[V\setminus R] + (c+1)[R]=[R]$. Thus $[R]$ is irreducible.
\end{proof}
Recall that a connected graph is {\em 2-connected} if it can not be disconnected by deleting a vertex. This property is characterised nicely in the corresponding graph lattice.
\begin{lem}\label{lem:2connectgraphlat}
The following are equivalent:
\begin{enumerate}[(i)]
\item The graph $G$ is 2-connected;
\item Every vertex $v\in V$ is irreducible;
\item The lattice $\Lambda(G)$ is indecomposable.
\end{enumerate}
\end{lem}
\begin{proof} The equivalence $(i)\Leftrightarrow (ii)$ follows from Lemma \ref{lem:irreducible}.

To show $(ii) \Rightarrow (iii)$, suppose $\Lambda(G)= L_1 \oplus L_2$. For any $v\in V$, $v$ irreducible implies $v\in L_1$ or $v\in L_2$. So if every $v$ is irreducible, then the sets $R_i=\{v | [v]\in L_i\}$ partition $V$. This gives $[R_1]+[R_2]=[V]=0$ and $[R_1]\cdot [R_2]=0$. It follows that $[R_1]=[R_2]=0$ and so, assuming $L_1\ne 0$, the partition is trivial, i.e. $R_1=V$ and $R_2=\emptyset$. This implies $L_1=\Lambda(G)$ and so $\Lambda(G)$ is indecomposable.

The implication $(iii)\Rightarrow (i)$ is shown as follows. Suppose $G$ is not 2-connected, so there is a vertex $v$ such that $G\setminus\{v\}$ is disconnected. Suppose this has components $G_1$ and $G_2$. Then we have $\Lambda(G)=L_1\oplus L_2$, where $L_1$ and $L_2$ are the sublattices spanned by the vertices of $G_1$ and $G_2$ respectively. So $\Lambda(G)$ is decomposable, if $G$ is not 2-connected.
\end{proof}

Finally, we use the above material to gain further information on the structure of a graph from its lattice. Recall that an edge, $e$, is a {\em cut-edge} if $G\setminus \{e\}$ is disconnected.

\begin{lem}\label{lem:cutedge}
Suppose that $G$ is 2-connected, contains no cut-edges and there is a vertex $v$ such that we can find $x,y\in \Lambda(G)$, with $v=x+y$ and $x\cdot y=-1$. Then there is a cut edge $e$ in $G\setminus \{v\}$ and if $R,S$ are the vertices of the two components of $(G\setminus \{v\})\setminus\{e\}$ then $\{x,y\}=\{[R]+v,[S]+v\}$. Furthermore, there are unique vertices $u_1,u_2\ne v$, with $x\cdot u_1=y\cdot u_2=1$, and any vertex $w \notin \{v,u_1,u_2\}$ satisfies $w\cdot x,w\cdot y\leq 0$.
\end{lem}
\begin{proof}
We will use the irreducibility of $v$ to show that $x$ and $y$ are irreducible. Suppose that we can write $x=z+w$ with $z\cdot w\geq 0$. Since $x\cdot y =-1$, it follows that $z\cdot y\geq 0$ or $w\cdot y\geq 0$. Without loss of generality, we may assume that $w\cdot y\geq 0$ and hence that $w\cdot (y+z)\geq 0$. Using the irreducibility of $v$, this implies either $w=0$ or $y+z=0$. If we assume $w\ne 0$, then it follows that $w=v$ and $y=-z$. Thus from $x\cdot y=-1$ and $w.z\geq 0$, we obtain $z\cdot (v+z)=1$ and  $v.z\geq 0$. Combining these shows that $\norm{z}\leq 1$. As $G$ contains no cut-edges, $\Lambda(G)$ does not contain any vectors of norm 1, so we can conclude that $z=0$. Therefore we have shown that $z=0$ or $w=0$, so we have shown that $x$ is irreducible.

\paragraph{} As $x$ and $y$ are irreducible, Lemma \ref{lem:irreducible} shows that they take the form $x=[R']$ and $y=[S']$, for some $R',S' \subseteq V$ satisfying
$$v=[R']+[S']=[R' \cup S'] + [R' \cap S'].$$
Since $x,y \ne 0$, we must have $R' \cup S'=V$ and $R' \cap S'=\{v\}$. Thus, $x=v+[R]$ and $y=v+[S]$, where $R$ and $S$ are disjoint and $R\cup S=V\setminus\{v\}$. Moreover, using irreducibility and Lemma \ref{lem:irreducible} again, it follows that  subgraphs induced by $R,S, R\cup\{v\}$ and $S\cup\{v\}$ must all be connected. Since $[R]+[S]=[V]-v=-v$, it follows that
\begin{equation*}
x\cdot y=\norm{v}+v\cdot ([R]+[S]) + [R]\cdot[S] =[R]\cdot[S]=-e(R,S)=-1.
\end{equation*}
This gives a unique edge from $R$ to $S$. The rest of the lemma follows by taking $u_1\in R$ and $u_2\in S$ to be the end points of this edge.
\end{proof}

\subsection{Change-Maker Lattices}\label{sec:CMlattices}
In this section, we define change-maker lattices and explore their properties. Change-maker lattices were defined by Greene in his solution to the lens space realization problem \cite{GreeneLRP} and work on the cabling conjecture \cite{greene2010space}. Our definition is the variant on these that arises in the case of half-integer surgery, (cf. \cite{Greene3Braid}).

\paragraph{}Let $\mathbb{Z}^{r+2}= \langle e_{-1}, e_0, \dotsc , e_r \rangle$, where $\{e_{-1}, \dotsc , e_r\}$ is an orthornormal basis. A {\em change-maker lattice} $L \subseteq \mathbb{Z}^{r+2}$ is defined to be the orthogonal complement,
$$L = \langle \rho , \sigma \rangle ^\bot,$$
where $\rho =e_{-1} - e_0$ and $\sigma = e_0+\sigma_1 e_1 + \dotsb \sigma_r e_r$ satisfies the {\em change-maker condition},
$$0\leq \sigma_1 \leq 1, \text{ and } \sigma_{i-1} \leq \sigma_i \leq \sigma_1 + \dotsb + \sigma_{i-1} +1,\text{ for } 1<i\leq r.$$
For the rest of the section, $L = \langle \rho , \sigma \rangle ^\bot$ will be a change-maker lattice. The case where $\sigma_1=0$ is a degenerate one. We will deal fully with this case at the end of the section in Lemma~\ref{lem:indecomp}. In the meantime, we will work under the assumption $\sigma_s\geq 1$ for $s\geq1$.

The change-maker condition is equivalent to the following combinatorial result.
\begin{prop}[Brown \cite{brown1961note}]
Let  $\sigma =\{\sigma_1, \dotsc , \sigma_s\}$, with $\sigma_1\leq \dotsb \leq \sigma_s$. There is $A\subseteq \{ 1, \dotsc , s\}$ such that $k=\sum_{i\in A} \sigma_i$, for every integer  $k$ with $0\leq k\leq \sigma_1 + \dotsb + \sigma_s$, if and only if $\sigma$ satisfies the change-maker condition.
\end{prop}

We say that $\sigma_s$ is {\em tight} if
$$\sigma_s=1+\sigma_1 + \dotsb + \sigma_{s-1},$$
and that $\sigma$ is {\em tight} if there is $s>1$ with $\sigma_s$ tight. Otherwise, we say that $\sigma$ is {\em slack}. The following variation on the preceding proposition will be useful.

\begin{lem}\label{lem:slacksum}
For any $s>1$ we can write $\sigma_s=\sum_{i \in A} \sigma_i$ with $A\subseteq \{0, \dotsc , s-1 \}$, and $1\in A$. Moreover, if $\sigma$ is slack we may assume that $0\notin A$.
\end{lem}
\begin{proof}
This can be done by induction on $s$.
\end{proof}

We establish the following notation for $v \in L$:
\begin{itemize}
 \item $ \supp(v) = \{i : v\cdot e_i \ne 0\}$
 \item $ \suppp(v) = \{i : v \cdot e_i> 0\}$
 \item $ \suppm(v) = \{i : v\cdot e_i < 0\}$
\end{itemize}
For $s \in \{1, \dotsc , r\}$, we construct a vector $v_k$ as follows. If $\sigma_s$ is tight, then set
$$ v_s = -e_s + e_{-1} + e_0 + \dotsb + e_{s-1}.$$
If $\sigma_s < \sigma_1 + \dotsb + \sigma_{s-1} + 1$, then there is $A_s \subset \{1, \dotsc , s-2\}$ such that $\sigma_s -\sigma_{s-1} = \sum_{i \in A_s} \sigma_i$.  Set
$$v_s = -e_s + e_{s-1} + \sum_{i \in A_s} e_i.$$
We call $S := \{v_1, \dotsc , v_r\}$ a {\em standard basis}. The requirements that $s-1 \in \supp(v_s)$ and that $0\notin \supp(v_s)$ unless $s$ is tight are not essential, and are chosen for convenience later in the paper.
\begin{lem}\label{lem:stdbasisisbasis}
A standard basis, $S$, is a basis for $L$.
\end{lem}
\begin{proof}
Since $k \in \supp(v_k)$ and $k \notin \supp(v_j)$ for $j<k$, we see that the the standard basis vectors are linearly independent.

Now consider $v\in L$. Suppose $k= \max \supp(v)\geq 1$. Set $v' = v + (v\cdot e_k)v_k$. By construction, $\max \supp(v') < k$. So to show that $S$ spans $L$, it suffices to show that there is no $v\neq 0$ with $\max \supp(v)<1$. Suppose we have such a $v$. Since $v\cdot\rho = 0$, we have $v=a(e_{-1} + e_0)$. This gives $\sigma\cdot v = a$. So $v \in L$ implies $v=0$. Thus the standard basis spans $L$.
\end{proof}

The following shows that if a lattice has a basis that looks like a standard basis, then it is a change-maker lattice.
\begin{lem}\label{lem:sufficentCMcondition}
Suppose we have a collection of vectors $\{w_1, \dotsc , w_r\} \subseteq \mathbb{Z}^{r+2}$ of the form
$w_s=-e_s+e_{s-1}+\sum_{i\in A_s} e_i$ with $A_s \subseteq \{-1,\dotsc , s-2\}$ and $-1\in A_s \Leftrightarrow 0\in A_s$, then there is an indecomposable change-maker lattice with $\{w_1, \dotsc , w_r\}$ as a basis.
\end{lem}
\begin{proof}
Take $\rho=e_{-1}-e_0$ as in the definition of a change-maker lattice. Define $\sigma_s'$, for $1\leq s\leq r$ inductively by
$$(\sigma_{-1}',\sigma_0')=(0,1)$$
and
$$\sigma_s' = \sigma_{s-1}' +\sum_{i\in A_s} \sigma_{i}'$$
By construction, $\{\sigma_1, \dotsc, \sigma_r'\}$ satisfies the change-maker condition. Set
$$\sigma'=e_0+ \sum_{s=1}^r \sigma_s'e_s,$$
then Lemma~\ref{lem:stdbasisisbasis} shows that
$$\langle w_1,\dotsc, w_r \rangle = \langle \sigma' , \rho \rangle^\bot,$$
as required.
\end{proof}

It will be useful to identify certain irreducible vectors in a change-maker lattice.
\begin{lem}\label{lem:cmirreducible}
Suppose we have $v= -e_k + \sum_{i\in A} e_i \in L$, where $A \subset \{-1, \dotsc , k-1\}$. Then $v$ is irreducible. In particular, all standard basis elements are irreducible. Additionally, if $\sigma_k$ is tight, then $v=-e_k + e_{k-1}+ \dotsb + e_2 + 2e_1$ is irreducible.
\end{lem}
\begin{proof}
Since we are assuming $\sigma_s\geq 1$ for $s\geq 1$, any non-zero $y\in L$ must have $\suppp(y)$ and $\suppm(y)$ non-empty.

First consider $v= -e_k + \sum_{i\in A} e_i \in L$. If there are $x,y \in L$ such that $v = x+y$ and $x\cdot y \geq 0$, then write $x=\sum x_i e_i$ and $y = \sum y_i e_i$. If $a,b \in \mathbb{Z}$ satisfy $a+b \in \{-1,0,1\}$, then $ab\leq 0$. Thus $x\cdot y =0$ and hence $(x_i,y_i) \in \{(\pm 1,0),(0,\pm 1),(0,0)\}$, for all $i$.

We have $x_i+y_i \geq 0$, for $i\ne k$ and $x_k+y_k = -1$. So, without loss of generality, we may assume that $x_k=-1$. It follows that $y_i \geq 0$ for all i. By the remark at the start of the proof, this implies $y=0$. Thus $v$ is irreducible.

\paragraph{} Now take $v=-e_k + e_{k-1}+ \dotsb + e_2 + 2e_1$. Suppose we write $v=x+y$, $x,y\in L$ with $x\cdot y\geq0$. Write $x=\sum x_i e_i$ and $y = \sum y_i e_i$, where $y_i=v\cdot e_i-x_i$. We have
\begin{equation}\label{eqn:irredcomp}
x\cdot y=\sum_{i=-1}^{r}x_i y_i=-x_{-1}^2-x_{0}^2+(2-x_1)x_1+\sum_{i=2}^{k-1}(1-x_i)x_i -(x_k+1)x_k -\sum_{i=k+1}^r x_i^2.
\end{equation}
The only term in the right hand side of (\ref{eqn:irredcomp}) which can possibly be positive is the $(2-x_1)x_1$ summand and this is at most 1. So $x\cdot y\geq 0$ implies that $(2-x_1)x_1=x_1y_1\in \{0,1\}$.

If $x_1 y_1 =0$, then it follows that $x_i y_i=0$ for all $i$. Thus we must have $\suppm(x)=\emptyset$ or $\suppm(y)=\emptyset$, which implies $x=0$ or $y=0$.

We will show that $x_1 y_1=1$ cannot occur. If $x_1 y_1 = 1$, then $x_1=y_1=1$ and there can be at most one $0\leq l\leq r$ with $x_l y_l<0$ and this has $x_l y_l =-1$. Since $x$ and $y$ are non-zero in this case, we require $\suppm(x)$ and $\suppm(y)$ non-empty, so such an $l$ exists. Since $x_l+y_l=0$ and $x_0=x_{-1}$, it follows that $l>k$. Thus, without loss of generality, $\suppm(x)=l$ and $\suppm(y)=k$. But we must also have $\suppp(x)\subseteq \{1, \dotsc , k-1\}$. This is a contradiction, since it implies $\sigma_l<\sigma_k$. Thus we have $x=0$ or $y=0$ and so $v$ is irreducible.
\end{proof}

Finally, we derive useful conditions for $L$ to be indecomposable.
\begin{lem}\label{lem:indecomp}
The following are equivalent:
\begin{enumerate}[(i)]
\item $L$ is indecomposable;
\item $\sigma_s\geq 1$ for all $s\geq 1$;
\item there is no $v\in L$ with $\norm{v}=1$.
\end{enumerate}
\end{lem}
\begin{proof}
The implications (i) $\Rightarrow$ (ii) and (ii) $\Leftrightarrow$ (iii) are straightforward; observe that if $\sigma_s=0$ for some $s$ then we get a $\mathbb{Z}$ summand of $L$. It remains only to check (ii) $\Rightarrow$ (i). Consider a standard basis, $\{v_1, \dotsc, v_r \}$. If $L=L_1 \oplus L_2$, then, by irreducibility, we must have $v_i \in L_1$ or $v_i \in L_2$ for every standard basis element. Suppose $v_1 \in L_1$. Suppose $L_2 \ne 0$, so there is $v_k \in L_2$. In fact, we may assume that $k$ is minimal with this property. This means that $v_k\cdot v_i=0$ for $1\leq i <k$. Since $\sigma_k>0$, we know $m=\min \supp(v_k)<k$. If $m>0$, then $v_k\cdot v_m=-1$. If $m\leq 0$, then $m=-1$ and $v_k \cdot v_1=1$. In either case, this is a contradiction. Thus, we have $L_2=0$ and it follows that $L$ is indecomposable.
\end{proof}

\section{The Goeritz matrix}\label{sec:goeritz}
 Given a diagram $D$ of a knot $K$, we get a division of the plane into connected regions. We may colour these regions black and white in a chessboard manner. There are two possible choices of colouring, and each gives an incidence number, $\mu(c)\in \{\pm 1\}$, at each crossing $c$ of $D$, as shown in Figure \ref{fig:incidencenumber}.
 \begin{figure}[h]
  \centering
  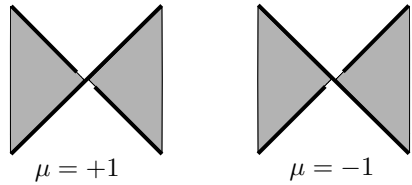
 \caption{The incidence number of a crossing.}
 \label{fig:incidencenumber}
\end{figure}
 We construct a planar graph, $\Gamma_D$, by drawing a vertex in each white region and an edge $e$ for every crossing $c$ between the two white regions it joins and we define $\mu(e):=\mu(c)$. We call this the {\em white graph} corresponding to $D$. This gives rise to a {\em Goeritz matrix}, $G_D=(G_{ij})$, defined by labeling the vertices of $\Gamma_D$, by $v_1,\dotsc , v_{r+1}$ and for $1\leq i,j \leq r$, set
 $$g_{ij}=\sum_{e \in E(v_i,v_j)}\mu(e)$$
 for $i\ne j$ and
 $$g_{ii} = - \sum_{e \in E(v_i, \Gamma_D\setminus v_i)}\mu(e)$$
 otherwise \cite[Chapter 9]{lickorish1997introduction}.

 \paragraph{} Although the Goeritz matrix depends on the choice of diagram, its determinant does not. Hence we can define the {\em determinant} of $K$ by
 $$\det K:= |\det G_D|.$$
 The other invariant we wish to compute from the Goeritz matrix is the knot signature.
 \begin{prop}[Gordon and Litherland \cite{gordon1978signature}]\label{prop:sigformula}
 Let $D$ be a diagram for a knot $K$ that is shaded so that there are $n_+$ positive crossings of incidence -1 and $n_-$ negative crossings of incidence +1. Let $\operatorname{sig}G_D$ be the signature of the Goeritz matrix $G_D$. Then the knot signature can be computed by the formula,
 \begin{equation*}
\sigma(K)=\operatorname{sig}G_D + n_- -n_+.
\end{equation*}
\end{prop}

 \paragraph{} Now suppose that $K$ is an alternating knot. If $D$ is any alternating diagram, then we may fix the colouring so that $\mu(c)=-1$ for all crossings. In this case, $G_D$ defines a positive-definite bilinear form. This in turn gives a lattice, $\Lambda_D$ which we will refer to as the {\em white lattice} of $D$. Observe that if $D$ is reduced (i.e. contains no nugatory crossings), then $\Gamma_D$ contains no self-loops or cut-edges and $\Lambda_D$ is isomorphic to the graph lattice $\Lambda(\Gamma_D)$. Moreover, if $D$ and $D'$ are any two reduced diagrams, then one can be obtained from another by a sequence of flypes \cite{Menasco93classification} and it is not hard to see that $\Lambda_D\cong \Lambda_{D'}$. The following lemma allows us to detect certain flypes algebraically and, provides an explicit isomorphism by relating the vertices of $\Gamma_D$ and $\Gamma_{D'}$.

 \begin{lem}\label{lem:flype} Let $D$ be a reduced alternating diagram. Suppose $\Gamma_D$ is 2-connected and has a vertex $v$, which can be written as $v=x+y$, for some $x,y \in \Lambda_D$ with $x\cdot y=-1$. Then there are unique vertices, $u_1,u_2\ne v$, satisfying $u_1\cdot x, u_2\cdot y>0$, and there is a flype to a diagram $D'$ and an isomorphism $\Lambda_{D'}\cong \Lambda_{D}$, such that the vertices of $\Gamma_{D'}$ are obtained from those of $\Gamma_D$ by replacing $v,u_1$ and $u_2$, with $x,y$ and $u_1+u_2$.
 \end{lem}
 \begin{proof}
 By Lemma \ref{lem:cutedge}, there is a cut edge, $e$, in $\Gamma_D\setminus\{v\}$ between $u_1$ and $u_2$ satisfying $u_1\cdot x=u_2\cdot y=1$. Thus there are subgraphs $G_1$ and $G_2$ such that $x=v+u_1+\sum_{z\in G_1}z$ and $y=v+u_2+\sum_{z\in G_2}z$. For $z_1\in G_1$, we have $z_1\cdot x=0$ and $z_1\cdot y= z_1\cdot v$. For $z_2\in G_2$, we have $z_2\cdot y=0$ and $z_2\cdot x=z_2\cdot v$. Thus, it follows that replacing $v,u_1$ and $u_2$ by $x,y$ and $u_1+u_2$ gives the set of vertices for the graph $\widetilde{G}$, obtained by replacing $v$ by two vertices with an edge $e'$ between them, in such a way that $\{e',e\}$ is a cut set, and then contracting $e$. This graph is planar, and can be drawn in the plane so that $\widetilde{G}=\Gamma_{D'}$ for $D'$, obtained by a flype about the crossing corresponding to the edge $e$ and the region corresponding to $v$. See Figure~\ref{fig:flype}.
 \begin{figure}[h]
  \centering
  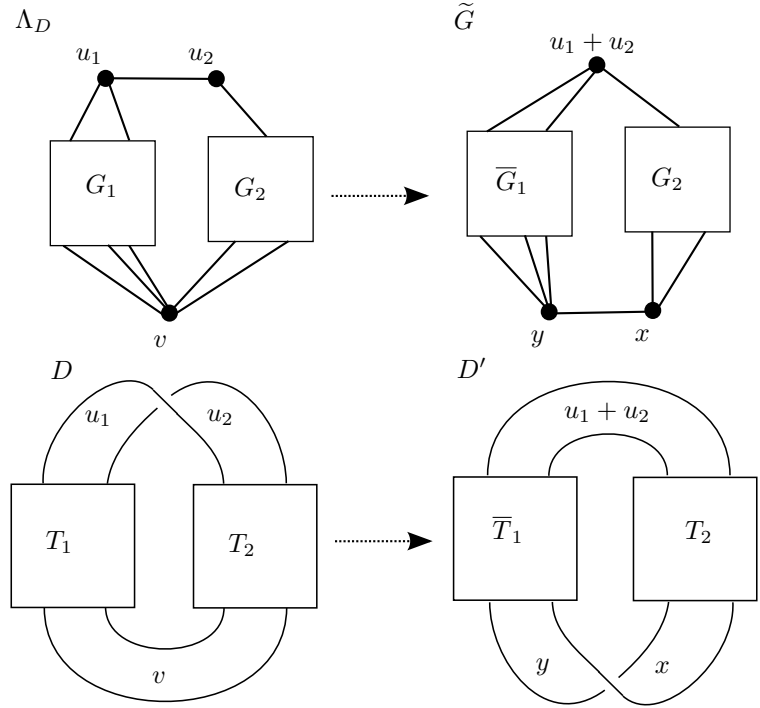
 \caption{The flype and new embeddings given by Lemma \ref{lem:flype}, where $\overline{G}_1$, is the reflection of $G_1$ about a vertical line and $\overline{T}_1$ is the corresponding tangle (which can also be obtained by rotating $T_1$ by an angle of $\pi$ about a vertical line).}
 \label{fig:flype}
\end{figure}
\end{proof}

\section{Branched double covers and surgery}\label{sec:branchedcovers}
Let $K\subset S^3$ be any knot and $\Sigma(K)$ its branched double cover. If $K$ has unknotting number one, then the Montesinos trick tells us that $\Sigma(K)$ can be constructed by half-integer surgery on a knot in $S^3$ \cite{montesinos1973variedades}. We state the precise version that we require (c.f. \cite{Greene3Braid,ozsvath2005knots}.)
\begin{prop}\label{prop:Montesinos}
Suppose that $K$ can be unknotted by changing a negative crossing to a positive one. Then there is a knot $\kappa \subset S^3$ such that $\Sigma(K)=S_{-\delta/2}^3(\kappa)$ where $\delta=(-1)^{\sigma(K)/2}\det(K)$.
\end{prop}

\paragraph{} The form of Greene's obstruction to half-integer surgery that will be most useful for our purposes is the following.
\begin{thm}[Greene, \cite{Greene3Braid}]\label{thm:GreeneLspaceversion}
Let $Y$ be an $L$-space bounding a sharp, simply-connected, positive-definite manifold, $X$, with intersection form $Q_X$ of rank $r$. Suppose also that there is a knot $\kappa\subset S^3$, such that $-Y=S_{\delta/2}^3(\kappa)$ and that
$$d(-Y,\spincs_0)=d(S_{\delta/2}^3(U),\spincs_0),$$
where $\spincs_0\in \spinc(Y)$ is the spin structure.
Then there is a change-maker lattice such that
$$Q_X \cong \langle \rho , \sigma \rangle ^\bot\subseteq \mathbb{Z}^{r+2}.$$
\end{thm}

As we show in Lemma~\ref{lem:deficiencies}, the condition on the $d$-invariants of the spin structures is redundant, since it is automatically satisfied when $Y$ is an $L$-space. Improvements in the understanding of the symmetries of $d$-invariants mean that the hypothesis that $Y$ is an $L$-space is not actually necessary for the above theorem to hold \cite{gibbons2013deficiency}. Since we are concerned with an application where $Y$ is an $L$-space, we will not require any such generalization.

\begin{lem}\label{lem:deficiencies}
Let $Y$ be a $L$-space with $-Y=S_{\delta/2}^3(\kappa)$. Let $\spincs_0$ be the spin structure, then
$$d(-Y,\spincs_0)=d(S_{\delta/2}^3(U),\spincs_0).$$
\end{lem}
\begin{proof}
First, we appeal to Corollary 1.4 of \cite{ozsvath2011rationalsurgery}, which shows that
$$\frac{\delta}{2}\geq 2g(\kappa)-1,$$
where $g(\kappa)$ is the genus of $\kappa$. Writing $\delta=4k\pm1$, this implies that
$$g(\kappa) \leq k.$$

There are affine correspondences \cite{ozsvath2011rationalsurgery},
$$\spinc(-Y)\longleftrightarrow \mathbb{Z}/\delta \mathbb{Z} \longleftrightarrow \spinc(S_{\delta/2}^3(U)),$$
under which the following formula holds for $0\leq i \leq \delta-1$,
$$d(-Y,i)-d(S_{\delta/2}^3(U),i)=-2\max \{V_{\lfloor \frac{i}{2} \rfloor}, H_{\lfloor\frac{i-\delta}{2}\rfloor}\},$$
where $H_i,V_i$ are integers satisfying $H_{-i}=V_{i}$ and $V_i=0$ for $i \geq g(\kappa)$ \cite{ni2010cosmetic}.

Under the correspondences in use, $\spincs_0$ is identified with $c= \frac{\delta+1}{2}$. Since
$$\left\lfloor \frac{c}{2} \right\rfloor =-\left\lfloor \frac{c-\delta}{2} \right\rfloor =k\geq g(\kappa),$$
the lemma follows.
\end{proof}

\section{Knots and change-maker lattices}\label{sec:mainsection}
In this section we will prove our main results. The objective is to show that if $D$ is a reduced alternating diagram with
$$\Lambda_D \cong \langle \rho , \sigma \rangle ^\bot\subseteq \mathbb{Z}^{r+2},$$
then $D$ contains an unknotting crossing.
Supposing we have such a $D$, we will use $L$ to denote the change-maker lattice and fix a choice of isomorphism
$$\iota_D : \Lambda_D \longrightarrow L= \langle \rho , \sigma \rangle ^\bot\subseteq \mathbb{Z}^{r+2}.$$
This gives us a distinguished collection of vectors in $L$ given by the image of the $r+1$ vertices of $\Gamma_D$. We call this collection $V_D$, and in an abuse of notation we sometimes fail to distinguish between a vertex of $\Gamma_D$ and the corresponding element in $V_D$.

\paragraph{} Since $D$ is minimal, $\Gamma_D$ contains no self loops or cut edges. In particular, there are no vectors of norm 1 in $\Lambda_D$, so $\Lambda_D$ is indecomposable by Lemma~\ref{lem:indecomp}. Lemma~\ref{lem:2connectgraphlat} implies that $\Gamma_D$ is 2-connected and any $v\in V_D$ is irreducible.

\paragraph{} It will be necessary for us to flype $D$ to obtain new minimal diagrams. In all cases, this flype will be an application of Lemma~\ref{lem:flype}. If $D'$ is the new diagram we obtain from such a flype, then Lemma~\ref{lem:flype} gives a natural choice of $V_{D'}\subset L$ and hence an isomorphism,
$$\iota_{D'}:\Lambda_{D'}\longrightarrow L.$$
We will implicitly use this choice of isomorphism and speak of $V_{D'}$ without ambiguity.

\begin{figure}[h]
  \centering
  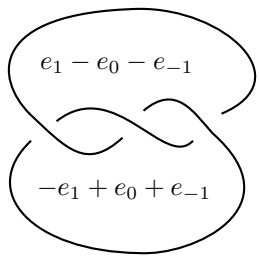
 \caption{The right-handed trefoil}
 \label{fig:trefoil}
\end{figure}

\begin{example}\label{rem:trefoilexample}
If $D$ is a reduced alternating diagram for the right-handed trefoil, then $\Gamma_D$ is a graph with two vertices and three edges and we have
$$\Lambda_D\cong\langle e_1+e_0, e_{-1}-e_0 \rangle ^\bot=\langle e_1-e_0-e_{-1} \rangle.$$
Thus $\Lambda_D$ is isomorphic to the unique indecomposable change-maker lattice of rank one and it has vertices $V_D=\{\pm (e_1-e_0-e_{-1})\}$. See Figure~\ref{fig:trefoil}. In fact, as $\langle e_1-e_0-e_{-1} \rangle$ has only two non-zero irreducible vectors, this is the only reduced alternating diagram with $\Lambda_D$ isomorphic to a change-maker lattice of rank one.
\end{example}

\paragraph{}We use the fact that $L$ is a change-maker lattice to deduce information about the elements of $V_D$. The following lemma serves as a useful sample calculation as well as being helpful in its own right.
\begin{lem}\label{lem:01splitting}
Let $x$ be a vertex or a sum of vertices. If there is $z=-e_g + \sum_{i \in A} e_i \in L$ with $A\subseteq \suppp(x)$ and $g\notin \suppp(x)$, then $(x-z)\cdot z\in \{-1,0 \}$.
\end{lem}
\begin{proof}
We apply Lemma~\ref{lem:usefulbound} to $(x-z)\cdot z$. This gives
\begin{equation*}
0\geq (x-z)\cdot z = -x\cdot e_g-1 + \sum_{i\in A} (x\cdot e_i -1).
\end{equation*}
By choice of $z$, $x\cdot e_g\leq 0$ and $x\cdot e_i\geq 1$ for all $i \in A$. So we can deduce that $(x-z)\cdot z\in \{0,-1\}$.
\end{proof}
It is worth noting that further information can be gleaned from the proof of the preceding Lemma. For example, we see that $x\cdot e_i=1$, for all but possibly one $i\in A$. Such arguments will be used frequently later in this paper. In these arguments, the vector $z$ will usually be provided by an application of Lemma~\ref{lem:slacksum}.

\subsection{Marked crossings}\label{sec:markedcrossings}
 In this section, we take the first steps towards understanding the structure of $D$. We will define marked crossings and show their existence. We also show that if there is more than one marked crossing, then any one of them is an unknotting crossing.

 \begin{lem}\label{lem:coefbound}
 Suppose $x = \sum_{v\in R}v \in L$ for $R\subseteq V_D$ is a sum of vertices, then
 \begin{enumerate}[(i)]
 \item $|x\cdot e_0|\leq 1$;
 \item $|x\cdot e_1|\leq 2$;
 \item if $x$ is irreducible with $x\cdot e_0\ne 0$ then $|x\cdot e_1|\leq 1$.
 \end{enumerate}
 \end{lem}
 \begin{proof} All parts are proved in a similar way.

 (i) Since $-x =\sum_{v\in V_D \setminus R}v$ is also a sum of vertices, we may assume that $x\cdot e_0\geq 0$. Let $g>0$ be minimal such that $x\cdot e_g\leq 0$. We may write $\sigma_g -1 = \sum_{i \in A} \sigma_i$ for some $A \subseteq \{1, \dotsc , g-1\}$. Let $z=-e_g + e_0 + e_{-1} + \sum_{i \in A} e_i$. By construction, $z \in L$. By Lemma~\ref{lem:usefulbound}, $(x-z)\cdot z\leq 0$. This gives
 $$0\geq(x-z)\cdot z \geq -1-x\cdot e_g + (x\cdot e_0 - 1) + (x\cdot e_{-1} - 1) \geq -1 + 2x\cdot e_0 -2.$$
 Thus $x\cdot e_0\leq \frac{3}{2}$.

 (ii) We may assume $x\cdot e_1\geq 0$. Let $g>1$ be minimal such that $x\cdot e_g\leq 0$. We may write $\sigma_g -1 = \sum_{i \in A} \sigma_i$ for some $A \subseteq \{1, \dotsc , g-1\}$. Let $z=-e_g + e_1 + \sum_{i \in A} e_i$. We have $z\cdot e_1=\epsilon \in \{1,2\}$. By Lemma~\ref{lem:usefulbound}, $(x-z)\cdot z\leq 0$. This gives
 $$0\geq (x-z)\cdot z\geq -1 + \epsilon(x\cdot e_1-\epsilon).$$
 Thus $x\cdot e_1\leq \epsilon + \frac{1}{\epsilon}\leq \frac{5}{2}$.

 (iii) Suppose now that $x$ is irreducible. We may assume that $x\cdot e_0 = 1$. Let $g>0$ be minimal such that $x\cdot e_g\leq 0$. If $g=1$, then let $z=-e_1 + e_0 + e_{-1}$. By irreducibility, it follows that either $x=z$ or
 $$(x-z)\cdot z= -(x\cdot e_1+1)\leq -1.$$
 This shows that $0\geq x\cdot e_1\geq -1$ in this case.
 Suppose now that $g>1$. We may write $\sigma_g -1 = \sum_{i \in A} \sigma_i$ for some $A \subseteq \{1, \dotsc , g-1\}$. If $1\in A$, set $z=-e_g + e_0 + e_{-1} + \sum_{i \in A} e_i$. If $1\notin A$, set $z=-e_g + e_{1} + \sum_{i \in A} e_i$. In either case $z\cdot e_1=1$. By irreducibility, it follows that either $x=z$ or
 $$ x\cdot e_1 - 1-(x\cdot e_g+1) \leq (x-z)\cdot z \leq -1.$$
 This shows that $x \cdot e_1 = 1$ in this case.
 \end{proof}

From part (i) of the above, it follows there is at most one vertex $v\in V_D$ with $v\cdot e_0>0$, and this necessarily has $v\cdot e_0=1$. It follows from the definition of a change-maker lattice that there is at least one such vertex. Similarly, there is a single vertex $w \in V_D$ with $w\cdot e_0<0$ and it satisfies $w\cdot e_0=-1$. We say that $v$ and $w$ are the {\em marker vertices} of $\Gamma_D$. The edges between $v$ and $w$ are the {\em marked edges}. The crossings in $D$ corresponding to the marked edges will be {\em marked crossings}.

\paragraph{} Returning to Example~\ref{rem:trefoilexample}, we see that every crossing of the trefoil in Figure~\ref{fig:trefoil} is a marked crossing.

\paragraph{}Note that replacing $\iota_D$ by $-\iota_D$ gives another isomorphism with the same marked crossings. For the purposes of notation, it will be convenient to fix this choice of sign in the next lemma.

\begin{rem}\label{rem:preservesmarked}
Suppose $v$ is a vertex that can be written $v=x+y$, with $x\cdot y=-1$. By Lemma~\ref{lem:flype}, this gives a cut edge, $e$ in $\Gamma_D\setminus\{v\}$ and a flype to a diagram with $x$ and $y$ as vertices. Observe that if $x\cdot e_0=0$ or $y\cdot e_0=0$, then the edge $e$ is not marked. This means that the flype can be chosen to fix the marked crossings which will again be marked crossings of the new embedding. In particular, such a flype commutes with the act of changing a marked crossing.
\end{rem}
\paragraph{} With this in mind, we make our first flypes and prove the existence of marked crossings.

\begin{lem}\label{lem:markedcexist}
Let $v$ and $w$ be the marker vertices with $v\cdot e_0=1$ and $w\cdot e_0=-1$. Up to choices of sign for $\iota_D$, we either have $v=-e_1+e_0+e_{-1}$ , or there is a flype, preserving marked crossings, which gives a diagram $D'$ with a marker vertex $v'=-e_1+e_0+e_{-1}$. In particular this means $D'$, and hence $D$, contains at least one marked crossing.
\end{lem}
\begin{proof}
First, we will show either $v\cdot e_1\leq0$ or $w\cdot e_1\geq 0$. Suppose $v\cdot e_1 =1$. By Lemma~\ref{lem:cmirreducible}, $v'=-e_1+e_0+e_{-1}$ is irreducible. Thus Lemma~\ref{lem:irreducible} implies there is $R \subset V_D$ such that $v'=\sum_{u\in R}u$. From $v'\cdot e_0=1$, it follows that $v\in R$ and $w\notin R$. Thus $v'-v+w$ is also a sum of vertices. Using part (ii) of Lemma~\ref{lem:coefbound}, we get
$$(v'-v+w)\cdot e_1=-2+w\cdot e_1\geq-2$$
This gives $w\cdot e_1\geq 0$.

\paragraph{} Thus up to replacing $\iota_D$ by $-\iota_D$, we may assume that $v\cdot e_1\leq 0$. As before, let $v'=-e_1+e_0+e_{-1}$. If $v\cdot e_1=-1$, then $v=v'$ as $v$ is irreducible. Assume that $v\cdot e_1=0$. In this case, $(v-v')\cdot v'=-1$. Applying Lemma~\ref{lem:flype} gives a flype to $D'$ with $(v-v')$ and $v'$ as vertices in $V_{D'}$. Since $(v-v')\cdot e_0=0$, this flype preserves marked crossings.
\paragraph{} The statement about existence follows by observing that since $|w\cdot e_1|\leq 1$ by part (iii) of Lemma~\ref{lem:coefbound} we have
$$v'\cdot w=-2-w\cdot e_1<0.$$
\end{proof}

\paragraph{} It follows that we may assume that $D$ has a marker vertex
$$v=-e_1+e_0+e_{-1}.$$
Consequently, there are at least one and at most three marked crossings in $D$. As we will see below, it is easy to show that $u(K)=1$, when there is more than one marked crossing. Further work will be required in the case of a single marked crossing.

\begin{lem}\label{lem:manymarkedcrossings}
Let $K'$ be the knot obtained by changing a marked crossing $c$ in $D$. Then $K'$ is almost alternating with $\det(K')=1$ and
\[\sigma(K')=
\begin{cases}
\sigma(K)+2 &\text{if $c$ is positive,}\\
\sigma(K)   &\text{if $c$ is negative.}
\end{cases}
\]
Moreover, if $D$ has more than one marked crossing, then $K'$ is the unknot and the diagram $D'$ obtained by changing $c$ is $\mathcal{C}_m$ or $\overline{\mathcal{C}_m}$ for some $m$.
\end{lem}
\begin{proof}
As in the statement of the lemma, let $D'$ be the diagram for $K'$ obtained by changing $c$ in $D$. The diagram $D'$ is almost-alternating by definition.

\paragraph{}The embedding of $\Lambda_D$ into $\mathbb{Z}^{r+2}$ gives a factorization of the Goeritz matrix, $G_D=AA^T$, where $A$ is the $r\times (r+2)$-matrix given by $A_{ij} =(v_i\cdot e_j)$, for some choice of $r$ vectors $\{v_1,\dotsc , v_r\}\subset V_D$. Let $C$ be the right-most $r\times r$ submatrix of $A$. Recall from Section~\ref{sec:CMlattices} that the change-maker lattice $L$ admits a standard basis. Let $\{w_1, \dotsc , w_r \}$ be a standard basis, where $w_i$ is the basis element with $w_i \cdot e_i = -1$. Let $C'$ be the $r \times r$-matrix $C'=(w_i \cdot e_j)_{1 \leq i,j \leq r}$. Since the lattices spanned by the rows of $C$ and $C'$ are isomorphic, we have $|\det(C)|=|\det(C')|$. As $\{w_1, \dotsc , w_r \}$ is a standard basis, $C'$ is triangular and all diagonal entries take the value -1. Therefore we have $|\det(C)|=|\det(C')|=1$.

\paragraph{}Let $w$ and $v$ be the marker vertices with $w\cdot e_0=-1$ and $v\cdot e_0=1$. Since $(w+e_0+e_{-1})\cdot(v-e_0-e_{-1})=w\cdot v+2$, we see that $G_{D'}=CC^T$ is a Goeritz matrix for $D'$. Therefore, $\det (K')=\det(C)^2 =1$.
\paragraph{} Now we compute the change in the knot signature. Observe that $G_{D'}$ is positive definite of rank $r$. Since the colouring on $D$ is such that every crossing has incidence number -1 and the corresponding Goeritz matrix is positive definite, Proposition~\ref{prop:sigformula} shows that
$$\sigma(K)=r-n,$$
where $n$ is the number of positive crossings in $D$. If $c$ is positive then $D'$ has $n-1$ positive crossings of incidence -1 and 1 negative crossing of incidence +1. If $c$ is negative, then $D'$ has $n$ positive crossings of incidence -1 and no negative crossings of incidence +1. Thus,
\[\sigma(K')=
\begin{cases}
r-n+2=\sigma(K)+2 &\text{if $c$ is positive,}\\
r-n=\sigma(K)   &\text{if $c$ is negative,}
\end{cases}
\]
as required.

\paragraph{} Suppose that $D$ has more than one marked crossing. We may assume that $D$ is a diagram with marker vertex $v=-e_1+e_0+e_{-1}$. Since $\norm{v}=3$, the marked crossings are adjacent in $D$. This allows us to perform a Reidemeister II move on $D'$ to obtain an alternating diagram, $D''$, for $K'$.
The white graph $\Gamma_{D''}$ is obtained by deleting two edges between the marker vertices of $\Gamma_D$. The determinant of an alternating knot is equal to the number of maximal spanning subtrees of the white graph of any alternating diagram \cite{crowell1959genus}. So as $\Gamma_D$ has no self-loops and $\det(K')=1$, it follows that $\Gamma_{D''}$ is a tree. Furthermore, as $\Gamma_D$ has no cut-edges, $\Gamma_{D''}$ must be a path whose endpoints were the marker vertices in $\Gamma_D$. Therefore, $D$ is a diagram of a clasp knot and $D'$ is $\mathcal{C}_m$ or its reflection $\overline{\mathcal{C}_m}$, depending on which marked crossing was changed, for some non-zero $m$.
\end{proof}

\begin{figure}
  \centering
  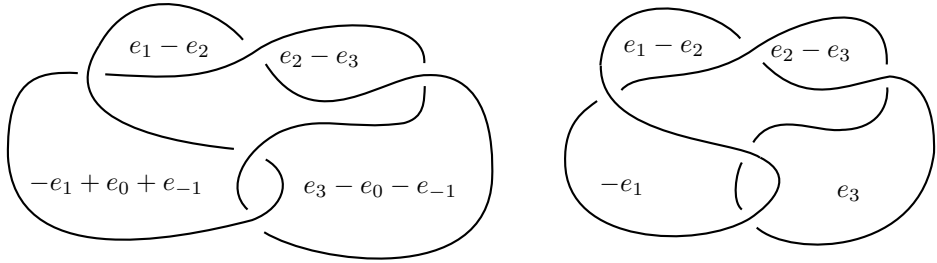
 \caption{Changing a marked crossing in the 5-crossing clasp knot.}
 \label{fig:clasp}
\end{figure}

\subsection{A single marked crossing}\label{sec:singlemarkedc}
We have already shown that $u(K)=1$ when $D$ has a multiple marked crossings. Dealing with the case of a single marked crossing requires more work. Assuming this is the case, we will perform further flypes to find diagrams with extra structure. The aim is to show we can flype so that after changing the marked crossing there is an obvious untongue or untwirl move that can be performed.

\paragraph{}By Lemma \ref{lem:markedcexist}, we may assume that $D$ is such that one of the marker vertices takes the form
$$v_1= -e_1+e_0+e_{-1}.$$
We are using $v_1$ to denote this marker vertex since it coincides with the notation for the standard basis of a change-maker lattice, as used in Section~\ref{sec:CMlattices}. We will write $w$ for the other marker vertex. In this section, we are assuming $w\cdot v_1=-1$, and in particular $w\cdot e_1=-1$. Since $v_1$ is a vertex of degree three it follows that there are other vertices pairing non-trivially with $v_1$. There are two cases to consider.
\begin{itemize}
 \item We say we are in {\em Situation A} if there are two vertices $u_1$ and $u_2$ with
 $$u_1\cdot v_1=u_2\cdot v_1=-1.$$
 From the comments after Lemma~\ref{lem:coefbound}, we necessarily have $u_1\cdot e_0=u_2\cdot e_0=0$ and $u_1\cdot e_1=u_2\cdot e_1=1$.
 \item We say we are in {\em Situation B} if there is a single vertex $u$ with $u\cdot v_1=2$. In this case, such a $u$ has $u\cdot e_0=0$ and $u\cdot e_1=2$.
\end{itemize}
In Situations A and B, we will call $u_1$ and $u_2$, and $u$, the {\em adjacent vertices} respectively.

\paragraph{} Let $k$ be maximal such that $\sigma_k=1$. If $k>1$, then for $2\leq a\leq k$ let $v_a$ be $v_a=-e_a+e_{a-1}\in L$. Again this notation coincides with that of the standard basis.

\begin{lem}
There is a sequence of flypes, preserving the marked crossing, to get a diagram $D'$ with $v_1,\dotsc , v_k$ as vertices.
\end{lem}
\begin{proof}
Suppose there is $1<l\leq k$, for which $v_l$ is not a vertex. Take $m>1$ to be minimal such that $v_m$ is not a vertex. By Lemma~\ref{lem:cmirreducible}, $v_m$ is irreducible, thus Lemma~\ref{lem:irreducible} shows that we may write $v_m=\sum_{v\in G_1}v$ for some subgraph $G_1\subseteq \Gamma_D$. As $v_{m-1}$ is a vertex and $v_m\cdot v_{m-1}=-1$, there is a vertex $x\in G_1$ with $x\cdot v_{m-1}=-1$. Since $x\in G_1$, this must satisfy $x\cdot v_m>0$ and, as $x$ is neither $v_m$ nor a cut vertex, we have $x\cdot v_m=1$. In particular, $(x-v_m)\cdot v_m=-1$. So by Lemma~\ref{lem:flype}, there is a flype to get a diagram $D'$ with $v_m$ and $x-v_m$ as vertices. Since $v_m\cdot v_i\leq0$ for $i<m$, this flype leaves $v_1, \dotsc , v_{m-1}$ as vertices and by Remark~\ref{rem:preservesmarked}, it can be chosen to preserve the marked crossing. Thus, we may flype so that all of the $v_1,\dotsc , v_k$ are vertices.
\end{proof}

If $D$ is such that $v_1,\dotsc , v_k$ are vertices, then we will say that $D$ is in {\em standard form}. By the previous lemma, we may now assume that $D$ is in standard form.

\paragraph{}If $k>1$, then we must be in Situation A and we can take $u_2=v_2$. Since $v_i$ can pair non-trivially with at most two vertices for $1<i\leq k$, it follows that
$$w\cdot e_1=\dotsb = w\cdot e_{k-1}=-1.$$
The inequality $u_1\cdot v_l\leq 0$ implies
$$u_1\cdot e_1= \dotsb = u_1\cdot e_k = 1.$$

Now we consider whether $\sigma$ is tight or slack, as defined in Section~\ref{sec:CMlattices}.
\begin{lem}\label{lem:slackvtight}
Assume $D$ is in standard form and we are in Situation A. Then the change-maker vector $\sigma$ is slack if, and only if, $\Gamma_D\setminus \{v_1,w\}$ is disconnected and $u_1$ and $u_2=v_2$ lie in separate components.
\end{lem}
\begin{proof}
Suppose $\sigma$ is slack. This implies $k>1$. Let $g>1$ be minimal such that $w\cdot e_g\geq 0$. By Lemma~\ref{lem:slacksum}, we may find $A\subseteq\{1,\dotsc , g-1\}$ with $1\in A$, such that $\sigma_g = \sum_{i\in A} \sigma_i$.
Consider $x = e_g-\sum_{i\in A}e_i\in L$. By Lemma~\ref{lem:cmirreducible} and Lemma~\ref{lem:irreducible}, this is irreducible and corresponds to a connected subgraph $G_1$ of $\Gamma_D$. Using Lemma~\ref{lem:usefulbound} combined with $x\cdot v_1=1$ and $x\cdot w\geq|A|=\norm{x}-1$, it follows that $x\cdot (v_1+w-x)=0$. In particular, this implies $v_1+w$ is reducible, since $x\ne v_1+w$. From Lemma~\ref{lem:irreducible}, it follows that $\{v_1,w\}$ is a cut set. As $x\cdot v_1>0$ and $x\cdot w>0$, we have $v_1, w \in G_1$. Using $x\cdot v_1=1=e(v_1,\Gamma_D\setminus G_1)$, we see that there is precisely one of the $u_i \in G_1$.

\paragraph{} For the converse, suppose there is $s>1$ with $\sigma_s=\sigma_{s-1}+\dotsb + \sigma_1 +1$.
Thus, by Lemma~\ref{lem:cmirreducible} and Lemma~\ref{lem:irreducible} there is a connected subgraph $G_2$, such that $y=\sum_{v\in G_2}v=-e_s + e_{s-1} + \dotsb +e_2 +2e_1$. Since $y\cdot e_1=2$, it follows that $v_1,w\notin G_2$ and $u_1,u_2 \in G_2$. Since $G_2$ is connected, it follows that $v_1$ and $w$ cannot separate the adjacent vertices.
\end{proof}

Armed with this information, we perform our final sequence of flypes. As ever, these will come from Lemma~\ref{lem:flype} and, by Remark~\ref{rem:preservesmarked}, can be chosen to commute with changing the marked crossing.

\begin{lem}\label{lem:tightslackflyping}
If $\sigma$ is tight and $w$ is not of the form
$$w=e_g-e_{g-1} - \dotsb - e_{-1},$$
then there is a flype to $D'$ in standard form such that we have marker vertex $w'\ne v_1$ satisfying $\max \supp(w')<\max \supp (w)$.

If $\sigma$ is slack and $u_1$ is not of the form
$$u_1= -e_h + e_{h-1} + \dotsb + e_1,$$
then there is a flype to $D'$ in standard form such that we have adjacent vertex $u_1'\ne v_2$ satisfying $\max \supp(u_1')<\max \supp (u_1)$.
\end{lem}

\begin{proof}
Suppose $\sigma$ is tight. Let $g$ be minimal such that $w\cdot e_g\geq 0$. By Lemma~\ref{lem:slacksum}, we may take $A\subseteq\{-1, 0, \dotsc , g-1\}$ such that $w'=e_g - \sum_{i \in A} e_i \in L$ and $1\in A$. We must have $\{-1,0, \dotsc ,k\}\subseteq A$. Otherwise, we could take $y=e_g-\sum_{i \in A'}e_i$, for some $A'\subseteq \{1, \dotsc , k-1\}$ with $1\in A'$. This would give $y\cdot v_1=1$ and $w\cdot y\geq \norm{y}-1$ thus allowing us to mimic the proof of Lemma~\ref{lem:slackvtight} to show that in Situation A, $\{v_1,w\}$ is a cut set separating $u_1$ and $u_2$. This would also give a contradiction in Situation B, since there is only one adjacent vertex.
\paragraph{} By Lemma~\ref{lem:01splitting}, $(w-w')\cdot w'\in\{-1,0\}$, so we consider the two possibilities separately. If $(w-w')\cdot w'=0$, then $w=w'$ by irreducibility of $w$. This implies
$$w=e_g-e_{g-1}-\dotsb - e_{-1}.$$
\paragraph{} If $(w-w')\cdot w'=-1$, then, by Lemma~\ref{lem:flype}, we may flype to get diagram $D'$ in standard form with $w'$ as a vertex of $D'$. In this case, we have
$$\max \supp(w')=g<\min \suppp (w)\leq \max \supp(w),$$
as required.

\paragraph{} Suppose $\sigma$ is slack, so we are necessarily in Situation A. Let $h$ be minimal such that $u_1\cdot e_h\leq 0$. Observe that $h>k$. By Lemma~\ref{lem:slacksum}, there is $A\subseteq \{1, \dotsc , h-1\}$ with $1\in A$ and $u_1'=-e_h + \sum_{i\in A} e_i \in L$.
\paragraph{}By Lemma~\ref{lem:01splitting}, $(u_1-u_1')\cdot u_1'\in\{0,-1\}$, so we consider the two cases separately. If $(u_1-u_1')\cdot u_1'=0$, then $u_1=u_1'$, by irreducibility of $u_1$. This implies
$$u_1=-e_h + e_{h-1} + \dotsb + e_1.$$
\paragraph{} Suppose $(u_1-u_1')\cdot u_1'=-1$. We may flype to a diagram $D'$ with $u_1'$ as a vertex. However, we need to show that $D'$ is in standard form. Since $v_1\cdot u_1'=-1$, it suffices to check that $v_i\cdot u_1'\leq 0$ for $1<i\leq k$. Suppose otherwise, so there is $1<b\leq k$ with $v_b\cdot u_1'=1$. By Lemma~\ref{lem:cutedge}, this implies that there is a cut-edge, $e$, incident to $v_b$ in $G\setminus u_1$. However, Lemma~\ref{lem:slackvtight} shows that $\{v_1,w\}$ is a cut set separating $u_1$ and $u_2$. In particular, $\{v_1,w\}$ also separates $u_1$ and $v_b$, as $u_2=v_2$ and $v_b$ are connected by $v_2, \dotsc , v_{b}$, a path of vertices of degree two. This implies that $e$ is actually a cut-edge in $G$, which is a contradiction as $D$ is reduced. Therefore, $D'$ is in standard form with $u_1'$ an adjacent vertex. In this case, we have
$$\max \supp(u_1')=h<\min \suppm (u_1)\leq \max \supp(u_1),$$
as required.
\end{proof}

\paragraph{}From the preceding lemma, it follows that we may flype so that diagram $D$ is in standard form with
$$w=e_g-e_{g-1} - \dotsb - e_{-1}$$
or
$$u_1= -e_h + e_{h-1} + \dotsb + e_1,$$
depending on whether $\sigma$ is tight or slack. As the following two lemmas will show, this will be the correct diagram for us to consider. Both proofs run along very similar lines, and like much of what has gone before, they make heavy use of Lemma~\ref{lem:usefulbound} applied to carefully chosen combinations of vectors.

\begin{lem}\label{lem:tightedgedetect}
Suppose $\sigma$ is tight and $D$ is in standard form with
$$w=e_g-e_{g-1} - \dotsb - e_{-1}.$$
Let $U=u_1+u_2$ if we are in Situation A, and $U=u$ if we are in Situation B. In either case, $U\cdot w<0$.
\end{lem}
\begin{proof}
Since $U\ne 0$, $\suppm(U)$ is nonempty, so we may take $m$ be minimal such that $U \cdot e_m<0$. Note the strict inequality here. We consider the cases $m\leq g$ and $m>g$ separately.

\paragraph{} Suppose $m\leq g$. By Lemma~\ref{lem:slacksum}, there is $A\subseteq\{-1,0,\dotsc, m-1\}$ with $1\in A$ and
$$z=e_m-\sum_{i\in A} e_i \in L.$$
Direct computation gives
$$(w-z)\cdot z= w\cdot e_m -1 \geq -2,$$
and
$$U\cdot z\leq -U\cdot e_1 + U\cdot e_m = -2 +U\cdot e_m \leq -3.$$
Combining these gives $(w-z-U)\cdot z>0$. However, Lemma~\ref{lem:usefulbound} yields
$$(w-z-U)\cdot z+w\cdot U=(w+U-(U+z))\cdot(z+U)\leq 0.$$
Therefore, $w\cdot U\leq -(w-z-U)\cdot z<0$, which is the required inequality.

\paragraph{} Now suppose $m>g$. In this case,
\begin{equation*}
w\cdot U=U\cdot e_g-\sum_{i=1}^{g-1}U\cdot e_i,
 \end{equation*}
and $U\cdot e_i\geq 0$ for $1\leq i \leq g$, so we wish to find a bound for $U\cdot e_g$. Let $l>g$ be minimal such that $U\cdot e_l\leq 0$. Such an $l$ exists and is at most $m$. By Lemma~\ref{lem:slacksum}, we may pick $A\subseteq\{g+1,\dotsc , l-1\}$, $B\subseteq \{-1, \dotsc , g-1\}$ with $1\in B$, and $\epsilon \in \{0,1\}$, such that
$$x=-e_l + \sum_{i \in B} e_i + \epsilon e_g + \sum_{j\in A} e_j \in L.$$
We write $C=\{-1, \dotsc , g-1\}\setminus B$. Note that $1\notin C$ and
$$x+w=-e_l + \sum_{j\in A}e_j + (1+\epsilon)e_g - \sum_{i\in C} e_i.$$
Since $w+U$ is a sum of vertices, we can apply Lemma~\ref{lem:usefulbound}, to get
\begin{align*}
0&\geq(w+U - (x+w))\cdot(x+w)=(U-x)\cdot(x+w)
\\ &\geq -1 + \sum_{j\in A}(U\cdot e_j-1) + (1+\epsilon)(U\cdot e_g-\epsilon) - \sum_{i\in C} U\cdot e_i
\\ &\geq -1 + (1+\epsilon)(U\cdot e_g-\epsilon) - \sum_{i\in C} U\cdot e_i.
\end{align*}
Here we are also using that $U\cdot e_j\geq1$ for $j\in A$.

As $1\notin C$, we also get the inequality,
\begin{equation*}
\sum_{i\in C} U\cdot e_i \leq \sum_{i=2}^{g-1}U\cdot e_i =\sum_{i=1}^{g-1}U\cdot e_i -2.
\end{equation*}
It follows from the above inequalities that
\begin{equation*}
U\cdot e_g \leq \frac{1}{1+\epsilon}(\sum_{i=1}^{g-1}U\cdot e_i -1) + \epsilon.
\end{equation*}
This allows us to compute
\begin{align*}
w\cdot U & =U\cdot e_g-\sum_{i=1}^{g-1}U\cdot e_i
\\&\leq \epsilon - \frac{1}{1+\epsilon}(\epsilon \sum_{i=1}^{g-1}U\cdot e_i + 1)
\\&\leq \epsilon - \frac{\epsilon U\cdot e_1 + 1}{1+\epsilon}
\\&= \epsilon - \frac{2\epsilon + 1}{1+\epsilon}.
\end{align*}
Since $\epsilon \in \{0,1\}$, this implies $w\cdot U<0$.
\end{proof}

\begin{lem}\label{lem:slackedgedetect}
Suppose $\sigma$ is slack and $D$ is in standard form with
$$u_1=-e_h+e_{h-1} + \dotsb + e_1.$$
If $w\cdot e_2=0$, then $w\cdot u_2=-1$. If $w\cdot e_2 =-1$, then $u_1\cdot w<0$.
\end{lem}
\begin{proof} As $\sigma$ is slack, we have $u_2=-e_2+e_1$. If $w\cdot e_2=0$, then $w\cdot u_2=-1$. So from now on we assume that $w\cdot e_2=-1$. The proof that $u_1\cdot w<0$ follows that of the preceding lemma closely.
\paragraph{}Take $m$ minimal such that $w\cdot e_m>0$. We consider the cases $m\leq h$ and $m>h$ separately.

\paragraph{} We deal with $m\leq h$ first. By Lemma~\ref{lem:slacksum}, we may choose $A\subseteq\{1, \dotsc, m-1\}$ such that $1\in A$ and
$$z=-e_m+\sum_{i\in A} e_i \in L.$$
If $2\in A$, then
$$(u_1-z)\cdot z=-u_1\cdot e_m -1\geq -2,$$
and
$$w\cdot z\leq w\cdot e_1 + w\cdot e_2 - w\cdot e_m \leq -3.$$
If $2\notin A$, then
$$(u_1+u_2-z)\cdot z=(u_1-z)\cdot z+u_2\cdot z= -u_1\cdot e_m \geq -1,$$
and
$$w\cdot z\leq -w\cdot e_1 +w\cdot e_m \leq -2.$$
In particular, the above calculations shows there is a choice, $U=u_1$ or $U=u_1+u_2$ such that
$$(U-z-w)\cdot z>0.$$
Applying Lemma~\ref{lem:usefulbound} yields
$$(U-z-w)\cdot z+w\cdot U=(U+w-(w+z))\cdot(w+z)\leq 0.$$
This gives the bound $w\cdot U<0$. The required inequality, $w\cdot u_1<0$, follows after observing that $u_2\cdot w=0$.

\paragraph{} Now suppose $m>h$. In this case,
\begin{equation*}
u_1\cdot w= - w\cdot e_h + \sum_{i=1}^{h-1} w\cdot e_i,
\end{equation*}
and $w\cdot e_i<0$ for $1\leq i\leq h$, so we need to bound $w\cdot e_h$. Let $l>h$ be minimal such that $w\cdot e_l\geq 0$. Such an $l$ exists and is at most $m$. By Lemma~\ref{lem:slacksum}, we may pick $A\subseteq\{h+1, \dotsc , l-1\}$, $B\subseteq \{1, \dotsc , h-1\}$ and $\epsilon \in \{0,1\}$ such that
\begin{equation*}
x= e_l - \sum_{j\in A} e_j - \epsilon e_h - \sum_{i\in B}e_i \in L.
\end{equation*}
Since $\sigma_2=1$, we may assume that $2\in B$.
If we write $C=\{1, \dotsc , h-1\}\setminus B$, then $2\notin C$ and
\begin{equation*}
x+u_1 = e_l - \sum_{j\in A} e_j - (\epsilon +1) e_h + \sum_{i\in C}e_i \in L.
\end{equation*}
Since $u_2+w+u_1$ is a sum of vertices, we can apply Lemma~\ref{lem:usefulbound} to get
\begin{align*}
0&\geq (u_2+ w + u_1 -(x+u_1))\cdot(x+u_1)
\\ &=(u_2 + w - x)\cdot(x+u_1)
\\ &\geq -1 - \sum_{j\in A}(w\cdot e_i-1) - (1+\epsilon)(w\cdot e_h-\epsilon) + \sum_{i\in C} (w\cdot e_i + u_2\cdot e_i)
\\ &\geq -1 - (1+\epsilon)(w\cdot e_h+\epsilon) + \sum_{i\in C} (w\cdot e_i + u_2\cdot e_i).
\end{align*}

Since $(w+u_2)\cdot e_1=0$, $(w+u_2)\cdot e_2=-2$ and $2\notin C$, we have
\begin{equation*}
\sum_{i\in C} (w\cdot e_i + u_2\cdot e_i) \geq \sum_{i=3}^{h-1} (w\cdot e_i + u_2\cdot e_i)=\sum_{i=1}^{h-1}w\cdot e_i+2.
\end{equation*}
Combining these inequalities gives
\begin{equation*}
-w\cdot e_h\leq \frac{-1}{1+\epsilon}(\sum_{i=1}^{h-1}w\cdot e_i +1)+\epsilon.
\end{equation*}
This allows us to compute
\begin{align*}
w\cdot u_1&=- w\cdot e_h + \sum_{i=1}^{h-1} w\cdot e_i
\\ &\leq \epsilon + \frac{1}{1+\epsilon}(\epsilon \sum_{i=1}^{h-1}w\cdot e_i -1)
\\ &\leq \epsilon + \frac{\epsilon w\cdot (e_1+e_2) -1}{1+\epsilon}
\\ &= \epsilon - \frac{2\epsilon + 1}{1+\epsilon}.
\end{align*}
Since $\epsilon \in \{0,1\}$, this gives $w\cdot u_1<0$.
\end{proof}

It will be useful to subdivide Situation A into two cases. If $w\cdot e_2=0$, then we say we are in {\em Situation A2}. Otherwise, we say that we are in {\em Situation A1}. Note that in Situation A2, $\sigma$ is necessarily slack and we must have $k=2$.
The results of this section can be summarised as the following.

\begin{prop}\label{prop:singlecrossingsummary}
If $K$ has a minimal diagram $D$ with a single marked crossing $c$, then there is a sequence of flypes to a minimal diagram, $D'$, in standard form which appears as shown in Figure~\ref{fig:localsubgraphs} in the neighbourhood of $c$. Each of the flypes fixes $c$ and hence commutes with changing $c$.
\end{prop}
\begin{proof}
By Lemma~\ref{lem:tightslackflyping}, there are a sequence of flypes fixing $c$ to get $D'$ in standard form with an adjacent vertex
$$u_1=-e_h+e_{h-1} + \dotsb + e_1$$
or a marker vertex
$$w=e_g-e_{g-1} - \dotsb - e_{-1},$$
depending on whether $\sigma$ is tight or slack.
In Situation A1, we can use Lemma~\ref{lem:slackedgedetect} or Lemma~\ref{lem:tightedgedetect} to deduce that $(u_2 + u_1)\cdot w<0$. We may assume $u_1\cdot w<0$. This means $v_1, w$ and $u_1$ form a triangle in $\Gamma_{D'}$.

Situation A2 is the case where $\norm{u_2}=2$ and $u_2\cdot w=u_2\cdot v_1=-1$. This means $v_1, w$ and $u_2$ form a triangle in $\Gamma_{D'}$ and there are no further edges from $u_2$.

In Situation B, Lemma~\ref{lem:tightedgedetect} gives $u\cdot w<0$. Thus $v_1, w$ and $u$ form a triangle in $\Gamma_{D'}$ with an additional edge between $u$ and $v_1$.
\end{proof}

\begin{figure}[h]
  \centering
  \def\svgwidth{\columnwidth}
  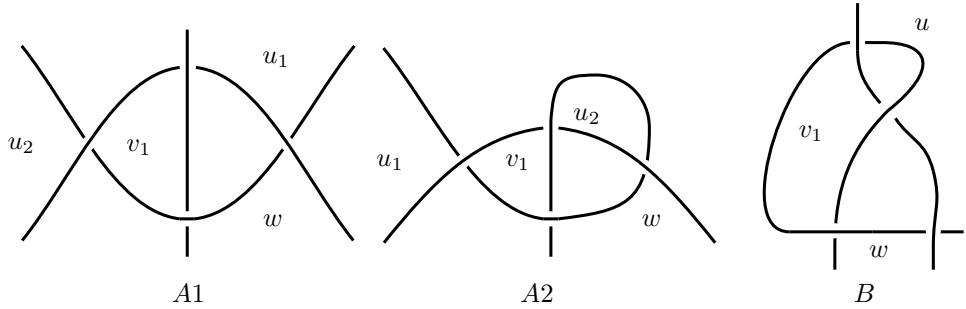
  \caption{In each of the Situations we may isotope to get a subtangle as pictured in the neighbourhood of the marked crossing.}
  \label{fig:localsubgraphs}
\end{figure}

\subsection{Changing marked crossings}\label{sec:smallerdiagram}
Now we study the effects of changing a marked crossing. We have already considered the case of multiple marked crossings in Lemma~\ref{lem:manymarkedcrossings}, so we suppose there is a single marked crossing, $c$. By the results of Section \ref{sec:singlemarkedc}, we may assume that we have flyped so that the neighbourhood of the marked crossing is as shown in Figure \ref{fig:localsubgraphs}. Let $K'$ be the knot obtained by changing the marked crossing. This crossing change gives an almost-alternating diagram $D'$. In Situation A1, there is an obvious untongue move and in Situations A2 and B there are obvious untwirl moves. Let $\widetilde{D}'$ be the new almost-alternating diagram obtained by performing these moves. There is a crossing, $\tilde{c}$, in $\widetilde{D}'$, such that changing $\tilde{c}$ gives an alternating diagram $\widetilde{D}$. In each case, the isotopies and crossing changes suggest how the embedding of $\Lambda_D$ can be modified to given an embedding of $\Lambda_{\widetilde{D}}$ into $\mathbb{Z}^{r+2}$. This is illustrated in Figures \ref{fig:sitA1embeddings}, \ref{fig:sitA2embeddings} and \ref{fig:sitBembeddings}.

These new embeddings will be used in the following lemma which will provide the induction step in a proof that $K'$ is the unknot.
\begin{lem}\label{lem:inductstep}
The diagram $\widetilde{D}$ is reduced and $\Lambda_{\widetilde{D}}$ can be embedded as a change-maker lattice with $\tilde{c}$ as a marked crossing.
\end{lem}
\begin{proof}
Let $S = \{v_1, \dotsc, v_r\} \subseteq \mathbb{Z}^{r+2}$ be a standard basis for $L$. Since each element of $S$ is irreducible, Lemma~\ref{lem:irreducible} implies that it can be written as a sum of vertices. Write
$$v_s=\sum_{x \in B_s}x,$$
where $B_s \subseteq V_D$.

Situations A1, A2, and B are treated separately, although for each one the proof is very similar. We will modify each set $B_s$ to give a new set $\widetilde{B}_s$ of vertices of $\widetilde{D}$. We will then observe that a subset of the collection $\{\tilde{v}_s:=\sum_{x\in \widetilde{B}_s}x\}$ satisfy the hypotheses of Lemma~\ref{lem:sufficentCMcondition} to get the desired conclusion.
\paragraph{Situation A1} If we write $w=-e_1 + \tilde{w}$ and $u_1=e_1 + \tilde{u}_1$, then $\Lambda_{\widetilde{D}}$ has an embedding into $\mathbb{Z}^{r+2}$ with vertices $V_{\widetilde{D}}$ obtained by replacing $\{v_1,w,u_1,u_2\}$ in $V_D$ by $\{v_1+u_2, \tilde{w}, \tilde{u}_1\}$. (See Figure~\ref{fig:sitA1embeddings})
\begin{figure}[h]
  \centering
  \def\svgwidth{300pt}
  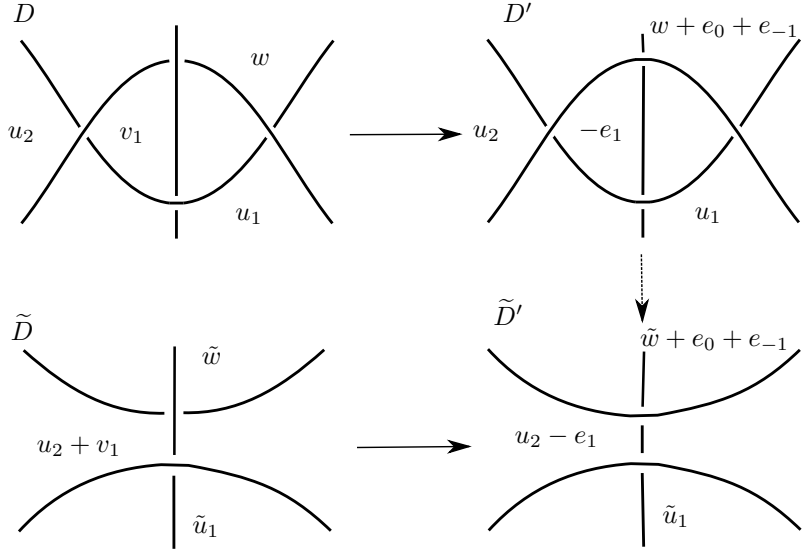
  \caption{An illustration of the embeddings for $\widetilde{D}$ in Situation A1}
  \label{fig:sitA1embeddings}
\end{figure}

\paragraph{} For $s\geq 2$, we obtain $\widetilde{B}_s$ by modifying $B_s$ in the following way. Replace any occurrences of $w$ by $\tilde{w}$, of $u_2$ by $v_1+u_2$ and of $u_1$ by $\tilde{u}_1$ and if $v_1\in B_s$, then discard it. Now define
$$\tilde{v}_s=\sum_{x\in \widetilde{B}_s}x.$$

By the definition of a standard basis, we have $v_s\cdot e_1, v_s\cdot e_0\geq0$ for $s\geq 2$. So, by considering the coefficients of the vertices of $\Lambda_D$, we see that $w\in B_s$ implies $v_1,u_1,u_2 \in B_s$. Similarly, for $s\geq 2$, $v_s\cdot e_0=1$ only if $v_s\cdot e_1=1$. So $v_1\in B_s$ implies $u_1,u_2\in B_s$. It follows that $\tilde{v}_2=-e_2+e_0+e_{-1}$ and, for $s\geq 3$,
 \[
\tilde{v}_s\cdot e_i
  \begin{cases}
   =v_s\cdot e_i            & \text{if } i\geq2 \\
   = 0       & \text{if } i=1\\
   \in\{0,1\}  & \text{if } i=0.
  \end{cases}
\]
We also have that $\tilde{v}_s\cdot e_0=\tilde{v}_s\cdot e_{-1}$. Therefore applying Lemma~\ref{lem:sufficentCMcondition} to $\{\tilde{v}_2, \dotsc , \tilde{v}_r\}$ shows that $\Lambda_{\widetilde{K}}$ embeds into
$$\langle e_{-1}, e_0, e_2, \dotsc , e_r\rangle \cong \mathbb{Z}^{r+1}$$
as an indecomposable change-maker lattice with $\tilde{c}$ as the marked crossing.

\paragraph{Situation A2} If we write $w=-e_1 + \tilde{w}$, then $\Lambda_{\widetilde{D}}$ has an embedding into $\mathbb{Z}^{r+2}$ with vertices $V_{\widetilde{D}}$ obtained by replacing $\{v_1,w,u_1,u_2\}\subseteq V_D$ by $\{v_1+u_1,\tilde{w}\}$. (See Figure~\ref{fig:sitA2embeddings})

\begin{figure}[h]
  \centering
  \def\svgwidth{300pt}
  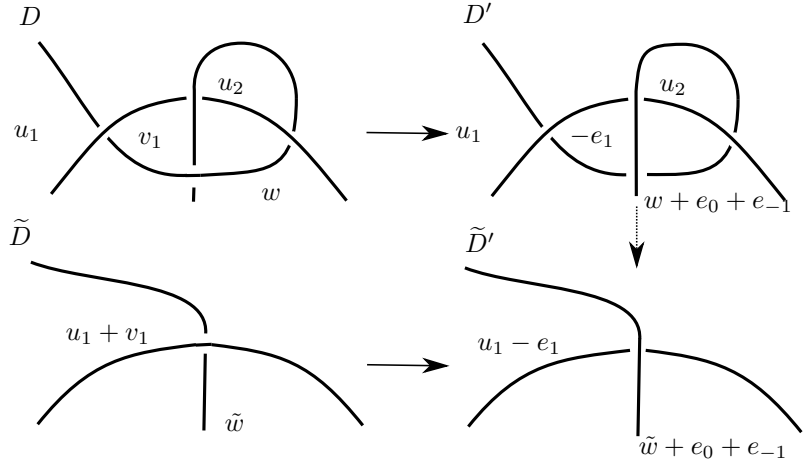
  \caption{An illustration of the embeddings for $\widetilde{D}$ in Situation A2}
  \label{fig:sitA2embeddings}
\end{figure}

For $s\geq 3$, we obtain $\widetilde{B}_s$ by modifying $B_s$ in the following way. Replace any occurrences of $w$ by $\tilde{w}$, of $u_1$ by $v_1+u_1$ and discard any occurrences of $v_1$ and $u_2$. Now define
$$\tilde{v}_s=\sum_{x\in \widetilde{B}_s}x.$$

By the definition of a standard basis, we have $v_s\cdot e_2, v_s\cdot e_1, v_s\cdot e_0\geq0$ for $s\geq 3$. So, by considering the coefficients of the vertices of $\Lambda_D$, we see that $w\in B_s$ implies $v_1,u_1,u_2 \in B_s$. Similarly, for $s\geq 3$, $v_s\cdot e_0=1$ only if $v_s\cdot e_1=1$. So $v_1\in B_s$ implies $u_1,u_2\in B_s$. It follows that $\tilde{v}_3=-e_3+e_0+e_{-1}$ and, for $s\geq 4$,
 \[
\tilde{v}_s\cdot e_i
  \begin{cases}
   =v_s\cdot e_i   & \text{if } i\geq 3\\
   = 0       & \text{if } i=1,2\\
   \in\{0,1\}  & \text{if } i=0.
  \end{cases}
\]
We also have that $\tilde{v}_s\cdot e_0=\tilde{v}_s\cdot e_{-1}$. Therefore applying Lemma~\ref{lem:sufficentCMcondition} to $\{\tilde{v}_3, \dotsc , \tilde{v}_r\}$ shows that $\Lambda_{\widetilde{K}}$ embeds into
$$\langle e_{-1}, e_0, e_3, \dotsc , e_r\rangle \cong \mathbb{Z}^{r}$$
as an indecomposable change-maker lattice with $\tilde{c}$ as the marked crossing.

\paragraph{Situation B} Recall that $u\cdot e_1=2$. If we write $w=-e_1 + \tilde{w}$ and $u=e_1 + \tilde{u}$, then $\Lambda_{\widetilde{D}}$ has an embedding into $\mathbb{Z}^{r+2}$ with vertices $V_{\widetilde{D}}$ obtained by replacing $\{v_1,w,u\}\subseteq V_D$ by $\{v_1 + \tilde{u}, \tilde{w}\}$. (See Figure~\ref{fig:sitBembeddings}.)

\begin{figure}[h]
  \centering
  \def\svgwidth{300pt}
  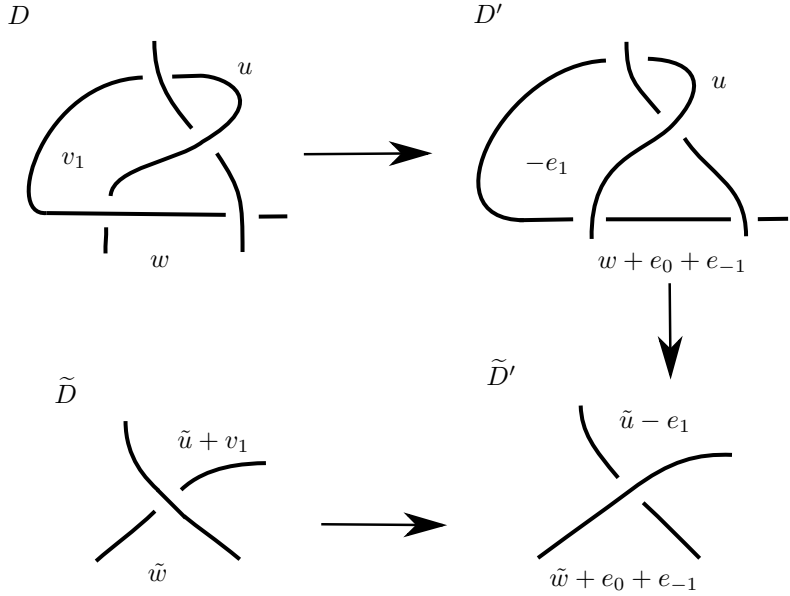
  \caption{An illustration of the embeddings for $\widetilde{D}$ in Situation B}
  \label{fig:sitBembeddings}
\end{figure}

For $s\geq 2$, we obtain $\widetilde{B}_s$ by modifying $B_s$ in the following way. Replace any occurrences of $w$ by $\tilde{w}$ and of $u$ by $v_1+\tilde{u}$. If $v_1\in B_s$, then discard it. Now define
$$\tilde{v}_s=\sum_{x\in \widetilde{B}_s}x.$$

By the definition of a standard basis, we have $v_s\cdot e_1, v_s\cdot e_0\geq0$ for $s\geq 2$. So, by considering the coefficients of the vertices of $\Lambda_D$, we see that $w\in B_s$ implies $v_1,u \in B_s$. Similarly, for $s\geq 2$, $v_s\cdot e_0=1$ only if $v_s\cdot e_1=1$. So $v_1$ or $u$ in $B_s$ implies both $v_1$ and $u$ are in $B_s$. It follows that $\tilde{v}_2=-e_2+e_0+e_{-1}$ and, for $s\geq 3$,
 \[
\tilde{v}_s\cdot e_i
  \begin{cases}
   =v_s\cdot e_i   & \text{if } i\geq2 \\
   = 0       & \text{if } i=1\\
   \in\{0,1\}  & \text{if } i=0.
  \end{cases}
\]
We also have that $\tilde{v}_s\cdot e_0=\tilde{v}_s\cdot e_{-1}$. Therefore applying Lemma~\ref{lem:sufficentCMcondition} to $\{\tilde{v}_2, \dotsc , \tilde{v}_r\}$ shows that $\Lambda_{\widetilde{D}}$ embeds into
$$\langle e_{-1}, e_0, e_2, \dotsc , e_r\rangle \cong \mathbb{Z}^{r+1}$$
as an indecomposable change-maker lattice with $\tilde{c}$ as the marked crossing.
\end{proof}

\subsection{The main results}\label{sec:mainresults}
We are now in a position to prove our main results.

\begin{proof}[Proof of Theorem~\ref{thm:markedmeansunknotting}]
By Lemma \ref{lem:markedcexist}, the embedding gives a marked crossing in $D$. We will prove by induction on the rank of $\Lambda_D$ that any marked crossing is an unknotting crossing. Suppose that $r=1$. As shown in Example~\ref{rem:trefoilexample}, this means $D$ is a minimal diagram of the right-handed trefoil. In particular, when $r=1$ every crossing in $D$ is both marked and an unknotting crossing.

Suppose now that $r>1$. If there is more than one marked crossing, then Lemma~\ref{lem:manymarkedcrossings} shows that any one of them is an unknotting crossing. So consider the case that there is a single marked crossing, $c$. Let $K'$ be the knot obtained by changing $c$. By Proposition~\ref{prop:singlecrossingsummary} and Lemma~\ref{lem:inductstep}, we see that $K'$ can be obtained by changing a marked crossing in a reduced alternating diagram $\widetilde{D}$ with the rank of $\Lambda_{\widetilde{D}}$ strictly less than $r$. Therefore, by the inductive hypothesis, $K'$ is the unknot.
\end{proof}

\begin{proof}[Proof of Theorem~\ref{thm:technical}]
The Montesinos trick, as stated in Proposition~\ref{prop:Montesinos}, gives $(i) \Rightarrow (ii)$. The implication $(iv)\Rightarrow (i)$ is trivial.
\paragraph{} Now we show $(ii)\Rightarrow (iii)$. If $\Sigma(K)=S^3_{-d/2}(\kappa)$, then $-\Sigma(K)=S^3_{d/2}(\overline{\kappa})$. Since $K$ is alternating, the work of Ozsv{\'a}th and Szab{\'o} shows that its branched double cover $\Sigma(K)$ is an $L$-space bounding a positive-definite sharp 4-manifold with intersection lattice isomorphic to $\Lambda_D$ \cite{ozsvath2005heegaard}.  Combining this with Lemma~\ref{lem:deficiencies} and Theorem~\ref{thm:GreeneLspaceversion} implies that $\Lambda_D$ is isomorphic to a change-maker lattice.

\paragraph{} To prove $(iii) \Rightarrow (iv)$, we take the marked crossing guaranteed by Theorem~\ref{thm:markedmeansunknotting}. Applying Lemma~\ref{lem:manymarkedcrossings} and the fact that the unknot has signature 0 shows that the marked crossing is of the required sign.
\end{proof}

\begin{proof}[Proof of Theorem~\ref{thm:main}]
This follows immediately when we observe that if $u(K)=1$, then least one of $K$ or its reflection $\overline{K}$ can be unknotted by a crossing change satisfying the conditions of Theorem~\ref{thm:technical}.
\end{proof}

\section{Almost-alternating diagrams of the unknot}\label{sec:aadiagrams}
In this section we reprove Tsukamoto's characterisation of almost-alternating diagrams of the unknot. It follows from the proof of Theorem~\ref{thm:markedmeansunknotting} that if an almost-alternating diagram of the unknot is obtained by changing a marked crossing, then it can be reduced by a sequence of flypes and untongue and untwirl moves. Thus, it suffices for us to show that any unknotting crossing in a reduced alternating knot diagram can be made into a marked crossing.

\paragraph{}Let $D$ be an alternating diagram with an unknotting crossing $c$, which we assume to be negative if $\sigma(D)=0$, and positive if $\sigma(D)=-2$. Consider the diagram $D'$ obtained by unknotting $c$, introducing a twirl and then changing a crossing back an alternating diagram, as in Figure~\ref{fig:swirlintro}.
\begin{figure}[h]
  \centering
  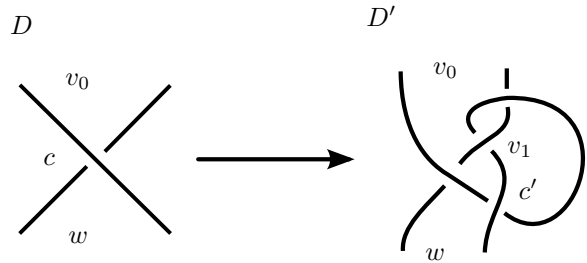
  \caption{Constructing $D'$ from $D$.}
  \label{fig:swirlintro}
\end{figure}
Since $c$ was an unknotting crossing, this has a new positive unknotting crossing $c'$ in the twirl. If $c$ is negative, then $D'$ has 3 new negative crossings. If $c$ is positive, then $D'$ has one new positive crossing. Thus Proposition~\ref{prop:sigformula} shows that $\sigma(D')=-2$.

Now let $D^{(n)}$, be the diagram obtained by iterating the process of unknotting and introducing twirls $n$ times. If the unknotting crossing was originally between region $v_0$ and $w$, then white graph $\Gamma_{D^{(n)}}$ is obtained from $\Gamma_D$ as in Figure~\ref{fig:iteratedswirl} below.
\begin{figure}[h]
  \centering
  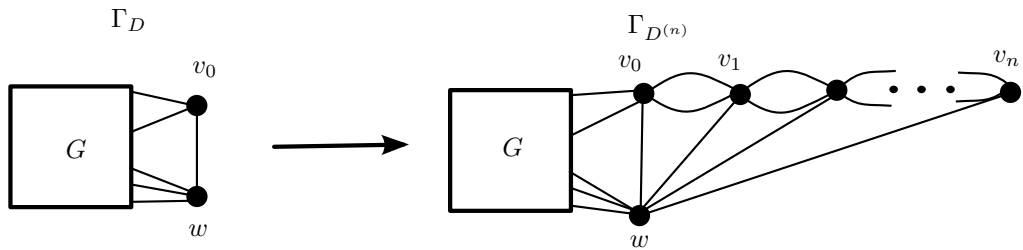
  \caption{Obtaining the white graph of $D^{(n)}$. The new unknotting crossing is between the regions labeled $w$ and $v_n$.}
  \label{fig:iteratedswirl}
\end{figure}

\begin{lem}\label{lem:manyswirls}
For sufficiently large $n$, the diagram $D^{(n)}$ has a unique unknotting crossing. Labelling the regions of $D^{(n)}$ as in Figure~\ref{fig:iteratedswirl}, this unknotting crossing is between the regions regions $v_n$ and $w$.
\end{lem}
\begin{proof}
Since $\sigma(D^{(n)})=-2$, Proposition~\ref{prop:sigformula} shows any unknotting crossing must be positive and the Goeritz matrix of the diagram obtained by changing it must be positive definite. Consider a diagram $\widetilde{D}$ obtained by changing a crossing which is not incident to any of the regions $v_1, \dotsc, v_n$. By discarding the region $w$, such a diagram has an $(n+r)\times (n+r)$ Goeritz matrix of the form given below
$$M_{n} =
  \begin{pmatrix}
   m_{1,1} & \cdots & m_{1,r} &    &         &    &   \\
   \vdots  & \ddots & \vdots  &    &         &    &   \\
   m_{r,1} & \cdots & m_{r,r} & -2 &         &    &   \\
           &        & -2      & 5  & -2      &    &   \\
           &        &         &-2  &  \ddots & -2 &   \\
           &        &         &    & -2      & 5  & -2\\
           &        &         &    &         & -2 &  3
  \end{pmatrix}$$
We will show that if we have chosen $n$ sufficiently large, then $\det M_n>1$ or $M_n$ is not positive definite. In either case, it follows that $\widetilde{D}$ is not the unknot. Let $d_k$, denote the determinant of the upper left $(r+k) \times (r+k)$-submatrix. By expanding the determinant along the bottom row of the submatrix we see that for $1\leq k \leq n-1$,
$$d_k=5d_{k-1}-4d_{k-2},$$
and
$$d_n=3d_{n-1}-4d_{n-2}.$$
Solving these linear recurrence relations, shows that there are constants $A$ and $B$ depending only on $d_0$ and $d_{-1}$, such that
$$d_k=4^kB+A,$$
for $-1\leq k \leq n-1$ and
$$d_n=2(4^{n-1})B-A.$$
If $M_n$ is positive definite, then $d_k>0$ for all $k$. This implies that
$$d_0+d_n=(2^{2n-1}+1)B>0,$$
and hence that $B>0$. Thus if we have chosen sufficiently large $n$, we have $d_n>1$. Since there are only finitely many possibilities for $\widetilde{D}$, it follows that for large enough $n$ any unknotting crossing of $D^{(n)}$ must be incident to at least one $v_1, \dotsc, v_n$.
\paragraph{} It is clear that a crossing between $v_i$ and $v_{i+1}$ cannot be an unknotting crossing; after changing such a crossing, we may perform a Reidermeister II move to obtain a non-trivial alternating diagram. Thus it remains only to show that the crossing between $v_s$ and $w$ is not an unknotting crossing for $1\leq s \leq n-1$. If we change such a crossing we get a diagram with Goeritz matrix of the form
$$M =
  \begin{pmatrix}
   m_{1,1}     & \cdots & m_{1,r+s-1}     &     &    &        &     &   \\
   \vdots      & \ddots & \vdots          &     &    &        &     &   \\
   m_{r+s-1,1} & \cdots & m_{r+s-1,r+s-1} & -2  &    &        &     &   \\
               &        & -2              & 3   & -2 &        &     &   \\
               &        &                 &-2   &  5 & -2     &     &   \\
               &        &                 &     & -2 & \ddots & -2  &   \\
               &        &                 &     &    & -2     &  5  & -2\\
               &        &                 &     &    &        & -2  &  3\\
  \end{pmatrix}$$
Where the upper left $(r+s)\times(r+s)$-submatrix is a Goeritz matrix for $D^{(s)}$. For $-1\leq k \leq n+r-s$, let $d_k$ be the determinant of the upper left $(r+s+k)\times (r+s+k)$-submatrix. The diagram $D^{(s)}$ can be unknotted by changing the crossing between $v_s$ and $w$. The Goeritz matrix we obtain by changing this crossing is positive and takes the form:
$$M' =
  \begin{pmatrix}
   m_{1,1}     & \cdots & m_{1,r+s-1}     &     \\
   \vdots      & \ddots & \vdots          &     \\
   m_{r+s-1,1} & \cdots & m_{r+s-1,r+s-1} & -2  \\
               &        & -2              & 1   \\

  \end{pmatrix}.$$
Thus, we have  $\det M'=d_0-2d_{-1}=1$.

Now consider the $(r+s+k)\times (r+s+k)$-submatrix of $M$. If we compute the determinant of this submatrix by expanding along the bottom row, we see that
$$d_k=5d_{k-1}-4d_{k-2},$$
for $1\leq k \leq n-s-1$, and
$$\det M=d_{n-s}=3d_{n-s-1}-4d_{n-s-2}.$$
By induction, we see that $d_k>2d_{k-1}>1$ for $0\leq k \leq n-s-1$. Therefore it follows that
$$\det M=3d_{n-s-1}-4d_{n-s-2}>2d_{n-s-2}>1.$$
Thus $M$ cannot be the Goeritz matrix of an unknot diagram. Hence we have shown that for large $n$ the only unknotting crossing lies between $w$ and $v_n$, as required.
\end{proof}

Theorem~\ref{thm:technical} implies that the Goeritz form $\Lambda_{D^{(n)}}$ is isomorphic to a change-maker lattice, and hence has a marked crossing. Since a marked crossing is an unknotting crossing, Lemma~\ref{lem:manyswirls} shows that  for large $n$ the marked crossing is between regions labelled $w$ and $v_n$. Since $D^{(n)}$ is obtained by adding twirls, we can modify this embedding as in Figure~\ref{fig:sitBembeddings} to obtain an embedding of $\Lambda_D$ into $\mathbb{Z}^{r+2}$. Lemma~\ref{lem:inductstep} shows that this new embedding makes $\Lambda_D$ into a change-maker lattice with $c$ as a marked crossing. In summary, this proves the following lemma.

\begin{lem}\label{lem:unknottingismarked}
Let $D$ be an alternating diagram with an unknotting crossing $c$, which is negative if $\sigma(D)=0$, and positive if $\sigma(D)=-2$. Then the Goeritz form $\Lambda_D$ can be embedded into $\mathbb{Z}^{r+2}$ as a change-maker lattice for which $c$ is a marked crossing.
\end{lem}

Now we complete our proof of Tsukamoto's classification of almost-alternating diagrams of the unknot.
\begin{proof}[Proof of Theorem~\ref{thm:tsukamoto}]
It follows from the proof of Theorem~\ref{thm:markedmeansunknotting} that a reduced almost-alternating diagram of the unknot obtained by changing a marked crossing admits a sequence of flypes, untongue and untwirl moves to $\mathcal{C}_m$ or $\overline{\mathcal{C}_m}$ for some non-zero $m$. Lemma~\ref{lem:unknottingismarked} shows that up to reflection any almost-alternating diagram of the unknot may be obtained by changing a marked crossing. Theorem~\ref{thm:tsukamoto} follows since there is a flype from $\overline{\mathcal{C}_m}$ to $\mathcal{C}_{-m}$.
\end{proof}

\section{Further questions}\label{sec:furtherqus}
Finally, we discuss various questions suggested by the results of this paper.

\subsection{Other surgery coefficients}
We now have a diagrammatic characterisation of when the  branched double cover of an alternating knot can be constructed by half-integer surgery on a knot in $S^3$. One may wish to ask whether something similar is true for other surgery coefficients. In this case, more general tangle replacements should replace the role of crossing changes.
\begin{ques}\label{ques:othercoef}
Let $K$ be an alternating knot with alternating diagram $D$ and suppose that $\Sigma(K)$ can be constructed by $p/q$-surgery on a knot in $S^3$ for some $p/q \in \mathbb{Q}$. Can a surgery of this slope be realised by some tangle replacement in $D$?
\end{ques}
Using Gibbons' generalisation of Greene's change-maker theorems \cite{gibbons2013deficiency} and ideas in this paper, it can be shown that Question~\ref{ques:othercoef} has a positive answer for $q>1$, i.e. when the surgery coefficient is not an integer. We will return to this question in future work \cite{mccoy2014noninteger}.

For $q=1$, Question~\ref{ques:othercoef} seems more difficult. This is reflected in the complexity of the combinatorics required by Greene in his solution to the lens space realization problem \cite{GreeneLRP}. This also makes $q=1$ the most interesting case of Question~\ref{ques:othercoef}, since a solution could well provide a new perspective on the lens space realization problem in terms of the 2-bridge links of which they are branched double covers.

\subsection{Unknotting crossings}
Combining Theorem~\ref{thm:markedmeansunknotting} and Lemma~\ref{lem:unknottingismarked} shows that a crossing in a reduced alternating diagram $D$ is an unknotting crossing if, and only if, it is a marked crossing for some isomorphism of $\Lambda_{D}$ or $\Lambda_{\overline{D}}$ to a change-maker lattice. So to find all unknotting crossings is a question of understanding all isomorphisms from $\Lambda_D$ to change-maker lattices. In all known examples, the isomorphism from $\Lambda_{D}$ to a change-maker lattice is essentially unique. This suggests the following conjecture.
\begin{conj}\label{conj:unknottingcrossing} If $K$ is an alternating knot with $u(K)=1$ and has an alternating diagram $D$ that contains more than one unknotting crossing of the same sign, then $K$ is a clasp knot.
\end{conj}
Conjecture~\ref{conj:unknottingcrossing} does not immediately follow from Theorem~\ref{thm:main} and Theorem~\ref{thm:tsukamoto} since there is the possibility that we might be able obtain two different almost-alternating diagrams of the unknot by changing crossings in $D$. Such diagrams would be built up by different sequences of tongue and swirl moves and would identify $\Lambda_D$ with a change-maker lattice in different ways.

Note that there are alternating knots with both positive and negative unknotting crossings.
\begin{example}
If $K$ is an alternating, amphichiral knot with $u(K)=1$, then there is $\kappa\subset S^3$, with
$$S_{-d/2}^3(\kappa)=\Sigma(K)=\Sigma(\overline{K})=S_{d/2}(\overline{\kappa}).$$
So Theorem~\ref{thm:technical} shows that any alternating diagram for $K$ contains both a positive and negative unknotting crossing. There are 6 examples of such knots with 12 or fewer crossings. One of these is the figure eight knot, $4_1$. The others are $6_3,8_9,10_{17},10_{35}$ and $12a_{1273}$ \cite{Knotinfo}.
\end{example}

\begin{example} The knot $8_{13}$ is an alternating knot which has both a negative and a positive unknotting crossing in any alternating diagram. This is an example which is not amphichiral \cite{Knotinfo}.
\end{example}

\subsection{Montesinos knots}\label{sec:montyknots}
Other knots for which the unknotting number has been extensively studied are the Montesinos knots. For definitions of Montesinos knots and some of their properties see \cite{BurdeZieschang}. It has been established that any Montesinos knot with four or more branches cannot have unknotting number one \cite{Boyer98Dehnfilling, Motegi96unlinking}.

For three branched Montesinos knots Torisu gives the following conjectural list \cite{Torisu96note}.
\begin{conj}\label{conj:Torisu}
Let $K$ be a Montesinos knot with three branches. Then $u(K)=1$ if, and only if, $K=\mathcal{M}(0; (p,r),(q,s),(2mn\pm 1, 2n^2))$, where $p,q,r,s,m$ and $n$ are non-zero integers, $m$ and $n$ are coprime and $ps+rq=1$.
\end{conj}
In support of his conjecture, Torisu shows the knots on his list are precisely those admitting unknotting crossings in standard Montesinos diagrams. Since the alternating Montesinos knots admit alternating Montesinos diagrams, Theorem~\ref{thm:main} resolves Conjecture~\ref{conj:Torisu} in the alternating case.

It is natural to wonder if the non-alternating case is also amenable to the methods of this paper. However, if $K$ is non-alternating, then only one of $\Sigma(K)$ or $\Sigma(\overline{K})$ is known to bound a sharp manifold \cite{ozsvath2003plumbed}, so Theorem~\ref{thm:GreeneLspaceversion} can only obstruct the sign of an unknotting crossing, cf. \cite{buck2013pretzel}.

\subsection{The Montesinos trick}
Theorem~\ref{thm:main} provides a converse to the Montesinos trick for alternating knots: an alternating knot has unknotting number one if, and only if, its branched double cover arises as half-integer surgery on a knot in $S^3$. It is natural to ask how far this converse extends.

\begin{ques}\label{ques:montyconverse}
For which classes of knots does the existence of $\kappa \subset S^3$ such that $\Sigma(K)=S^3_{d/2}(\kappa)$ imply that $u(K)=1$?
\end{ques}
Given the discussion in Section~\ref{sec:montyknots}, it seems likely that Question~\ref{ques:montyconverse} has a positive answer for Montesinos knots.

\paragraph{} There are, however, knots for which the converse to the Montesinos trick does not hold. The author is grateful to Liam Watson for suggesting the following examples.
\begin{example}\label{exam:montyconverse}
Let $T_{r,s}$ denote the $(r,s)$-torus knot in $S^3$. Recall that the unknotting number is given by the formula \cite{kronheimer93gauge, rasmussen10khovanov}:
$$u(T_{r,s})=\frac{1}{2}(r-1)(s-1).$$
So the only torus knot with unknotting number one is $T_{3,2}$. However, for odd $q>1$, we have \cite[Proposition 2]{watson2010khovanov}:
$$S^3_{\pm 1/2}(T_{2,q})= \Sigma (T_{q,4q\mp 1}).$$
So for odd $q>1$, torus knots of the form $T_{q,4q\pm 1}$ provide a supply of knots whose branched double cover can be constructed by half-integer surgery and which have unknotting number greater than one.
\end{example}
One conjecture governing Seifert fibred surgeries is the following.
\begin{conj}[Seifert fibering conjecture]\label{conj:seifertfibring}
If non-integer surgery on a knot in $S^3$ is Seifert fibred, then the knot is a torus knot or a cable of a torus knot.
\end{conj}
It can be shown that Conjecture~\ref{conj:seifertfibring} implies that the only torus knots whose branched double covers can be constructed by half-integer surgery on knots in $S^3$ are $T_{3,2}$, which is alternating, and those listed in Example~\ref{exam:montyconverse}. Thus it seems likely that any further examples of the failure of the converse to the Montesinos trick cannot be provided by torus knots.

\bibliographystyle{plain}
\bibliography{Unknottingv4}

\begin{thebibliography}{10}

\bibitem{bernhard1994unknotting}
James~A. Bernhard.
\newblock Unknotting numbers and minimal knot diagrams.
\newblock {\em J. Knot Theory Ramifications}, 3(1):1--5, 1994.

\bibitem{bleiler1984note}
Steven~A. Bleiler.
\newblock A note on unknotting number.
\newblock {\em Math. Proc. Cambridge Philos. Soc.}, 96(3):469--471, 1984.

\bibitem{Boyer98Dehnfilling}
S.~Boyer and X.~Zhang.
\newblock On {C}uller-{S}halen seminorms and {D}ehn filling.
\newblock {\em Ann. of Math. (2)}, 148(3):737--801, 1998.

\bibitem{brown1961note}
J.~L. Brown, Jr.
\newblock Note on complete sequences of integers.
\newblock {\em Amer. Math. Monthly}, 68:557--560, 1961.

\bibitem{buck2013pretzel}
Dorothy Buck, Julian Gibbons, and Eric Staron.
\newblock Pretzel knots with unknotting number one.
\newblock {\em Comm. Anal. Geom.}, 21(2):365--408, 2013.

\bibitem{BurdeZieschang}
Gerhard Burde and Heiner Zieschang.
\newblock {\em Knots}, volume~5 of {\em de Gruyter Studies in Mathematics}.
\newblock Walter de Gruyter \& Co., Berlin, second edition, 2003.

\bibitem{Knotinfo}
J.C. Cha and Livingston C.
\newblock Knotinfo: table of knot invariants.
\newblock \url{http://www.indiana.edu/~knotinfo}, 2014.

\bibitem{cochran1986unknotting}
T.~D. Cochran and W.~B.~R. Lickorish.
\newblock Unknotting information from {$4$}-manifolds.
\newblock {\em Trans. Amer. Math. Soc.}, 297(1):125--142, 1986.

\bibitem{crowell1959genus}
Richard Crowell.
\newblock Genus of alternating link types.
\newblock {\em Ann. of Math. (2)}, 69:258--275, 1959.

\bibitem{cglscyclic}
Marc Culler, C.~McA. Gordon, J.~Luecke, and Peter~B. Shalen.
\newblock Dehn surgery on knots.
\newblock {\em Ann. of Math.}, 125(2):237--300, 1987.

\bibitem{donaldson1983application}
S.~K. Donaldson.
\newblock An application of gauge theory to four-dimensional topology.
\newblock {\em J. Differential Geom.}, 18(2):279--315, 1983.

\bibitem{gibbons2013deficiency}
Julian Gibbons.
\newblock Deficiency symmetries of surgeries in ${S}^3$.
\newblock {\em arXiv:1304.0367}, 2013.

\bibitem{gordon1978signature}
C.~McA. Gordon and R.~A. Litherland.
\newblock On the signature of a link.
\newblock {\em Invent. Math.}, 47(1):53--69, 1978.

\bibitem{gordon2006knots}
C.~McA. Gordon and John Luecke.
\newblock Knots with unknotting number 1 and essential {C}onway spheres.
\newblock {\em Algebr. Geom. Topol.}, 6:2051--2116, 2006.

\bibitem{greene2010space}
Joshua Greene.
\newblock {L}-space surgeries, genus bounds, and the cabling conjecture.
\newblock {\em arXiv:1009.1130}, 2010.

\bibitem{GreeneLRP}
Joshua Greene.
\newblock The lens space realization problem.
\newblock {\em Ann. of Math. (2)}, 177(2):449--511, 2013.

\bibitem{Greene3Braid}
Joshua Greene.
\newblock Donaldson's theorem, {H}eegaard {F}loer homology, and knots with
  unknotting number one.
\newblock {\em Adv. Math.}, 255(0):672 -- 705, 2014.

\bibitem{jablan1998unknotting}
Slavik~V. Jablan.
\newblock Unknotting number and {$\infty$}-unknotting number of a knot.
\newblock {\em Filomat}, (12, part 1):113--120, 1998.

\bibitem{Unknottwobridge}
Taizo Kanenobu and Hitoshi Murakami.
\newblock Two-bridge knots with unknotting number one.
\newblock {\em Proc. Amer. Math. Soc.}, 98(3):499--502, 1986.

\bibitem{kauffman87state}
Louis~H. Kauffman.
\newblock State models and the {J}ones polynomial.
\newblock {\em Topology}, 26(3):395--407, 1987.

\bibitem{kohn1991two}
Peter Kohn.
\newblock Two-bridge links with unlinking number one.
\newblock {\em Proc. Amer. Math. Soc.}, 113(4):1135--1147, 1991.

\bibitem{kronheimer93gauge}
P.~B. Kronheimer and T.~S. Mrowka.
\newblock Gauge theory for embedded surfaces. {I}.
\newblock {\em Topology}, 32(4):773--826, 1993.

\bibitem{lickorish1997introduction}
W.B.~Raymond Lickorish.
\newblock {\em An introduction to knot theory}.
\newblock Springer, 1997.

\bibitem{mccoy2014noninteger}
Duncan McCoy.
\newblock Non-integer surgery and branched double covers of alternating knots.
\newblock {\em In preparation}, 2014.

\bibitem{Menasco93classification}
William Menasco and Morwen Thistlethwaite.
\newblock The classification of alternating links.
\newblock {\em Ann. of Math. (2)}, 138(1):113--171, 1993.

\bibitem{montesinos1973variedades}
Jos\'{e}~M. Montesinos.
\newblock Variedades de seifert que son recubricadores ciclicos rami cados de
  dos hojas.
\newblock {\em Boletino Soc. Mat. Mexicana}, 18:1--32, 1973.

\bibitem{Motegi96unlinking}
K.~Motegi.
\newblock A note on unlinking numbers of {M}ontesinos links.
\newblock {\em Rev. Mat. Univ. Complut. Madrid}, 9(1):151--164, 1996.

\bibitem{murasugi1965certain}
Kunio Murasugi.
\newblock On a certain numerical invariant of link types.
\newblock {\em Trans. Amer. Math. Soc.}, 117:387--422, 1965.

\bibitem{murasugi86jones}
Kunio Murasugi.
\newblock Jones polynomials of alternating links.
\newblock {\em Trans. Amer. Math. Soc.}, 295(1):147--174, 1986.

\bibitem{ni2010cosmetic}
Yi~Ni and Zhongtao Wu.
\newblock Cosmetic surgeries on knots in ${S}^3$.
\newblock {\em J. Reine Angew. Math}, (to appear), 2013.

\bibitem{owens2008unknotting}
Brendan Owens.
\newblock Unknotting information from {H}eegaard {F}loer homology.
\newblock {\em Adv. Math.}, 217(5):2353--2376, 2008.

\bibitem{ozsvath2003plumbed}
Peter Ozsv{\'a}th and Zolt{\'a}n Szab{\'o}.
\newblock On the {F}loer homology of plumbed three-manifolds.
\newblock {\em Geom. Topol.}, 7:185--224 (electronic), 2003.

\bibitem{ozsvath2005knots}
Peter Ozsv{\'a}th and Zolt{\'a}n Szab{\'o}.
\newblock Knots with unknotting number one and {H}eegaard {F}loer homology.
\newblock {\em Topology}, 44(4):705--745, 2005.

\bibitem{ozsvath2005heegaard}
Peter Ozsv{\'a}th and Zolt{\'a}n Szab{\'o}.
\newblock On the {H}eegaard {F}loer homology of branched double-covers.
\newblock {\em Adv. Math.}, 194(1):1--33, 2005.

\bibitem{ozsvath2011rationalsurgery}
Peter~S. Ozsv{\'a}th and Zolt{\'a}n Szab{\'o}.
\newblock Knot {F}loer homology and rational surgeries.
\newblock {\em Algebr. Geom. Topol.}, 11(1):1--68, 2011.

\bibitem{rasmussen10khovanov}
Jacob Rasmussen.
\newblock Khovanov homology and the slice genus.
\newblock {\em Invent. Math.}, 182(2):419--447, 2010.

\bibitem{stoimenow2001examples}
A.~Stoimenow.
\newblock Some examples related to 4-genera, unknotting numbers and knot
  polynomials.
\newblock {\em J. London Math. Soc. (2)}, 63(2):487--500, 2001.

\bibitem{stoimenow2004polynomial}
A.~Stoimenow.
\newblock Polynomial values, the linking form and unknotting numbers.
\newblock {\em Math. Res. Lett.}, 11(5-6):755--769, 2004.

\bibitem{tait1876knots}
P.~G. Tait.
\newblock On knots.
\newblock {\em Proc. Royal Soc. Edinburgh}, 97(9):306--317, 1876-7.

\bibitem{thistlethwaite88alternating}
Morwen~B. Thistlethwaite.
\newblock Kauffman's polynomial and alternating links.
\newblock {\em Topology}, 27(3):311--318, 1988.

\bibitem{Torisu96note}
Ichiro Torisu.
\newblock A note on {M}ontesinos links with unlinking number one (conjectures
  and partial solutions).
\newblock {\em Kobe J. Math.}, 13(2):167--175, 1996.

\bibitem{trotter1962homology}
H.~F. Trotter.
\newblock Homology of group systems with applications to knot theory.
\newblock {\em Ann. of Math. (2)}, 76:464--498, 1962.

\bibitem{tsukamoto2009almost}
Tatsuya Tsukamoto.
\newblock The almost alternating diagrams of the trivial knot.
\newblock {\em J. Topol.}, 2(1):77--104, 2009.

\bibitem{watson2010khovanov}
Liam Watson.
\newblock A remark on {K}hovanov homology and two-fold branched covers.
\newblock {\em Pacific J. Math.}, 245(2):373--380, 2010.

\end{thebibliography}
\end{document}